\documentclass[a4paper,reqno,11pt]{amsart}

\usepackage[left=2.8cm,right=2.8cm,top=2.8cm,bottom=2.8cm]{geometry}
\usepackage[utf8]{inputenc}	
\usepackage{indentfirst} %Per avere il rientro nella prima riga di ogni capitolo.
\usepackage{amsmath,amsfonts,amssymb,,mathrsfs}
\usepackage{bbm}
\usepackage[colorlinks=true]{hyperref}
\usepackage{enumerate}
\usepackage{graphicx}
\usepackage[dvipsnames]{xcolor}
\usepackage{todonotes}
\usepackage[shortlabels]{enumitem}
\usepackage{commath} % norm{} and abs{}
\usepackage{amsthm}
\usepackage{esint}
\usepackage{parskip}

\renewcommand{\textcolor}[2]{#2}

\theoremstyle{theorem}
\newtheorem{definition}{Definition}[section]
\theoremstyle{theorem}

\theoremstyle{theorem}
\newtheorem{theorem}[definition]{Theorem}
\theoremstyle{theorem}
\newtheorem{lemma}[definition]{Lemma}
\theoremstyle{remark}

\theoremstyle{theorem}
\newtheorem{corollary}[definition]{Corollary}
\theoremstyle{theorem}
\newtheorem{proposition}[definition]{Proposition}
\theoremstyle{theorem}

\numberwithin{equation}{section}
\allowdisplaybreaks

\def\XXint#1#2#3{{\setbox0=\hbox{$#1{#2#3}{\int}$}
      \vcenter{\hbox{$#2#3$}}\kern-.5\wd0}}

\definecolor{ao}{rgb}{0.0, 0.5, 0.0}

\title[Effective theories]{Effective theories for
incompressible magnetoelastic
shallow shells.}
\date{\today}

\author{
  Emanuele Tasso \and Tobias Unterberger}

\AtEndDocument{
	\bigskip
	\footnotesize{
		\noindent

		\noindent
		Institute of Analysis and Scientific Computing,
		TU Wien,
		
		\noindent
		Wiedner Hauptstra\ss e 8-10, 1040 Vienna, Austria.
		
		\noindent
		E-mails: emanuele.tasso@asc.tuwien.ac.at, tobias.unterberger@asc.tuwien.ac.at
		
}}

\begin{document}
	
	\begin{abstract}
	   We characterize the asymptotic behaviour, in the sense of $\Gamma$-convergence, of a thin magnetoelastic shallow shell. The compactness is achieved up to rigid motions. For deformations, it relies on an approximation by rigid movements, whereas for magnetizations it is based on a careful consideration of the geometry of the deformed domain. 
       
       The result is obtained by a combination of variational methods ($\Gamma$-convergence) with degree theory, fixed-point and geometrical arguments. The proof strategy relies on an adaptation of an analogous result for incompressible magnetoelastic plates from M. Bresciani \cite{BrescianiIncompressible} and an application of results by I.Velcic on elastic shallow shells \cite{Velcic}.
  
		\medskip
  
		\noindent
		{\it 2020 Mathematics Subject Classification:}

		\smallskip
		\noindent
		{\it Keywords and phrases: $\Gamma$-convergence, Magnetoelasticity, Dimension reduction}

	\end{abstract}
	
	\maketitle
	
%%%%%%%%%%%%%%%%%%%%%%%%%%%%%%%%%%%%%%%%%%
 
\section{Introduction}
\label{s:1}
While the study of thin structures in elasticity by $\Gamma$-convergence is by now a classical result (see, e.g., \cite{dav}); reduced theories in magnetoelasticity are much harder to identify owing to the mixed Euler-Lagrangian structure of the energy. In particular, the only available contributions in the large-strain setting are those by M.Bresciani in \cite{BrescianiIncompressible, BreKru21, BreDavKru21}.\\
In this paper we extend the analysis performed in \cite{BrescianiIncompressible} for magnetoelastic plates to the setting of magnetoelastic shallow shells. The latter ones are two-dimensional surfaces with a tubular neighbourhood where the curvature of the mean surface is small with respect to the size of the mean surface. 

To be precise, consider a cylindrical set $\Omega = \omega \times (-\frac{1}{2},\frac{1}{2})$ with $\omega \subseteq \mathbb{R}^2$. An associated plate of height $h > 0$ would be given by $\Omega_h = \omega \times (-\frac{h}{2},\frac{h}{2})$, whereas the corresponding shallow shell $\hat{\Omega}_h$ with thickness $h > 0$ is defined through the map
\begin{equation*}
    (x', x_3) \mapsto (x', h\theta(x')) + hx_3n_h(x') \ \text{for} \ (x',x_3) \in \Omega,
\end{equation*}
where $\theta$ is a smooth function and $n_h$ is the unit normal vector to the midsurface of $\hat{\Omega}_h$.

Its associated energy $\mathcal{E}_h(u_h,m_h)$ then depends on the elastic deformation $u_h: \hat{\Omega}_h \to \mathbb{R}^3$ and on the magnetization $m_h:u_h(\hat{\Omega}_h) \to \mathbb{S}^2$, and is given by
\begin{equation}
    \begin{aligned}
        \mathcal{E}_h(u_h,m_h) := \ &\frac{1}{h^\beta}\int_{\hat{\Omega}_h}W^{inc}(\nabla u_h, m_h\circ u_h) \ dX\\
        + \ &\alpha\int_{u_h(\hat{\Omega}_h)}|\nabla m_h|^2 \ d\xi + \frac{1}{2}\int_{\mathbb{R}^3}|\nabla \psi_{m_h}|^2 \ d\xi.
    \end{aligned}
    \label{1.1}
\end{equation}

The first term describes the \textit{elastic energy}. We consider deformations $u_h$ in the Sobolev-space $W^{1,p}(\hat{\Omega}_h;\mathbb{R}^3)$ and assume $\beta > 2p$, as well as $p > 3$ to ensure that every deformation $u_h$ admits a continuous representative and that the deformed set $u_h(\hat{\Omega}_h)$ is defined in a meaningful way. The incompressible elastic energy density $W^{inc}:\mathbb{R}^{3\times 3}\times \mathbb{S}^2 \to \mathbb{R} $, is defined as $W^{inc}(F,\nu) = W(F,\nu)$ if $\det(F) = 1$ and $W^{inc}(F,\nu) = + \infty$ otherwise, where the nonlinear elastic energy density $W:\mathbb{R}^{3\times 3}\times \mathbb{S}^2 \to \mathbb{R}$ is frame-indifferent, normalized and satisfies suitable growth and regularity assumptions, which will be specified in Section \ref{s:3}. Note that we restrict our analysis to incompressible deformations, i.e. such that $\det \nabla u_h \equiv 1$ in $\hat{\Omega}_h$.

The second term in \eqref{1.1} models the \textit{exchange energy}, with the exchange constant $\alpha > 0$. It arises from the pairwise interaction of magnetic dipoles and favours their alignment. Magnetizations $m_h$ are Sobolev maps $W^{1,2}$, which take values in the unit sphere $\mathbb{S}^2$ due to a necessary modelling assumption called the \textit{magnetic saturation constraint}.  However, since $u_h$ is not a homeomorphism, $u_h(\hat{\Omega}_h)$ is not necessarily open. We therefore replace it with a suitable open set $\hat{\Omega}_h^{u_h}$ specified in Section \ref{s:3} and assume $m_h \in W^{1,2}(\hat{\Omega}_h^{u_h};\mathbb{S}^2)$.

The third term in \eqref{1.1} is called the \textit{magnetostatic energy}, where $\psi_{m_h}:\mathbb{R}^3 \to \mathbb{R}$ represents the stray field potential, defined as the weak solution to the following Maxwell equation:
\begin{equation*}
    \Delta \psi_{m_h} = \text{div}(\chi_{\hat{\Omega}_h^{u_h}}m_h) \ \text{in} \ \mathbb{R}^3.
\end{equation*}

The main contribution of this work is the characterization of the limiting behaviour of the rescaled magnetoelastic energy $h^{-1}\mathcal{E}_h$. We present below a simplified statement and refer to Theorems \ref{thm1} and \ref{thm2} for a precise formulation.

\textbf{Theorem.}\textit{
    The asymptotic behaviour of the functionals $(h^{-1}\mathcal{E}_h)$, as $h \to 0^+$, is described, in the sense of $\Gamma$-convergence, by the functional:
    \begin{equation}
        \begin{aligned}
            E(u,v,\lambda) := \ &\frac{1}{2}\int_\omega Q_2^{inc}(\text{sym}(\nabla' u + \nabla'v \otimes \nabla'\theta), \lambda) \ dx' + \frac{1}{24}\int_\omega Q_2^{inc}((\nabla')^2v,\lambda) \ dx'\\
            + \ &\alpha\int_\omega|\nabla'\lambda|^2 \ dx' + \frac{1}{2}\int_\omega|(\lambda)^3|^2 \ dx',
        \end{aligned}
        \label{1.3}
    \end{equation}
    defined for $u \in W^{1,2}(\omega;\mathbb{R}^2)$, $v \in W^{2,2}(\omega)$ and $\lambda \in  W^{1,2}(\omega;\mathbb{S}^2)$. 
}

Given $\nu \in \mathbb{S}^2$, the function $Q_2^{inc}$ is a quadratic form on $\mathbb{R}^{2\times 2}$ constructed from the second-order approximation of $W(.,\nu)$ close to the identity. $\nabla'$ denotes the gradient with respect to the variable $x' \in \omega$ and $(\lambda)^3$ is the third component of $\lambda$. The maps $u$ and $v$ denote in-plane and out-of-plane displacements, respectively, whereas $\lambda$ is the limiting magnetization. Note that the key difference from the limiting functional for incompressible magnetoelastic plates in \cite{BrescianiIncompressible} is the dependence of $Q_2^{inc}$ on $\nabla'v \otimes \nabla'\theta$, thus keeping track of the original curved geometry.

As usual for $\Gamma$-convergence, the proof of the main result can be split into three parts. We prove

\begin{enumerate}
    \item \textbf{Compactness:} Compactness of sequences of states with equibounded energies.
    \item \textbf{Liminf inequality:} We show an explicit bound from below for the asymptotic behaviour of the energies on pairs $(u_h,m_h)$ with equibounded energies.
    \item \textbf{Recovery sequence:} We prove that the lower bound in (2) is optimal by constructing recovery sequences so that such limiting values are attained.
\end{enumerate}

The compactness and the liminf inequality are summarized in Theorem \ref{thm1}. To prove the compactness, we first derive an approximation by rotations for the gradient of deformations $\nabla u_h$ in Section \ref{s:4}. For this, starting from the seminal rigidity estimate by G.Friesecke, R.James, S.Müller \cite{OGrigidity}, we provide an approximation of $\nabla u_h$ by Sobolev maps $R_h$ taking values in the set of proper rotations $SO(3)$. The basic case with quadratic growth was studied in \cite{Velcic, shells} (I.Velcic, M.Mora et al.). We adapt these results to suit the different growth assumption for magnetoelastic materials (see \ref{W3}). We conclude the section by proving compactness of deformations $u_h$ and magnetizations $m_h$. For this proof, the key observation is that the mapping from the plate to the shallow shell $\Omega_h \to \hat{\Omega}_h$ is a $C^1$-diffeomorphism, as shown in \cite{Velcic}. This allows to deduce the compactness in the setting of shallow shells directly from that of magnetoelastic plates. The argument for the elastic energy in the setting of magnetoelastic plates is based on an adaptation of standard techniques of dimension reduction in nonlinear elasticity \cite{FJM-hierachy} to the different growth assumption. For the compactness of magnetizations on plates, a novel approach based on careful considerations of the geometry of the deformed sets is applied. Using a uniform convergence estimate of the deformations $u_h$ together with elementary properties of the topological degree, it is shown that the deformed sets $u_h(\Omega_h)$ contain a cylinder with a height of order $h$. On this cylinder, standard methods are applied to deduce compactness. To conclude, it is shown that the locally identifed convergence is globally well-defined. For details we refer to \cite[Proposition 4.2]{BrescianiIncompressible}.

The liminf inequality is proven in Section \ref{s:7}. Similarily to the compactness, it is directly deduced from the analogous result on magnetoelastic plates. As mentioned in \cite{BrescianiIncompressible}, the proof for the elastic energy is adapted from \cite{FJM-hierachy}, accounting for the incompressibility by deploying the same strategy as in \cite{ContiEstimate, LiChe12}. Similarly to the proof of the compactness, the liminf inequality for the exchange energy is obtained by considering a family of cylinders contained in the deformed sets that exhaust them, in the sense of measure. For the magnetostatic energy, the previously derived geometric considerations about the deformed domains are combined with an adaptation of results in \cite{thinfilms} to deduce first the compactness of the corresponding solutions and then the convergence of the magnetostatic energies.

The recovery sequence is the subject of Theorem \ref{thm2} and is shown in Section \ref{s:8} by adapting arguments from \cite{BrescianiIncompressible}, showing the existence of a smooth recovery sequence and then arguing by density. First, an ansatz adapted from \cite{Velcic} for the recovery sequence is made. Then, the incompressibility constraint is incorporated by techniques developed in \cite{LiChe12}. The constructed deformations are perturbations of the identity, which then, by a result in \cite{CiraletElasticity} are globally injective. Once more, the fact that the mapping between the plate and the shallow shell is a $C^1$-diffeomorphism is crucial for the adaptation. The argument is finished analogously to \cite[Proposition 5.1]{BrescianiIncompressible}: The convergence of the elastic energies rely on arguments from \cite{FJM-hierachy}, for the exchange energy it is straightforward and for the magnetostatic energy similar arguments as for the lower bound are made. 

To summarize, in this paper, starting from previous analogous results in nonlinear elasticity for shallow shells and in magnetoelasticity for plates, we have identified an effective two-dimensional model for magnetoelastic shallow shells. Further research directions could be the case of magnetoelastic general shells, where possibly strong curvature effects can come into play, as well as an extention to the setting in which deformations are in $W^{1,p}$ with $p \geq 2$, where more delicate measure-theoretical techniques would be needed (see \cite{Bre22}). Note that anisotropic exchange terms as well as DMI terms could be included in our analysis at no cost, for they would behave as continuous perturbations.

%%%%%%%%%%%%%%%%%%%%%%%%%%%%%%%%%%%%%%%%%%%%%%%%%%%%%%%%%%%%%%%%%%%%

\section{Notation and preliminary results}
\label{s:2}
In this section we introduce the relevant notation as well as we recall some results concerning the topological degree and  an estimate of rigidity type.

\subsection*{Notation}

\begin{itemize}
    
    \item By $B_r(x)$ we denote the ball of radius $r$ centered at $x$. If $ x \in \mathbb{R}^2$ we write $D_r(x) := B_r(x)$.
    
    \item Given a vector $x = (x_1, x_2, x_3) \in \mathbb{R}^3$, we write $x' = (x_1, x_2)$ and also $(x)^3$ for the third component.

     \item By $\pi$ we denote the projection $\pi: x \mapsto x'$.

     \item By $\mathbb{S}^2$ we denote the unit sphere.
    
    \item For $x \in \mathbb{R}$, $\bigl \lceil x \bigr \rceil$ denotes the ceiling function of $x$, which maps $x$ onto the smallest integer greater or equal to $x$.
    
    \item For a $n\times n$-matrix $A \in \mathbb{R}^{n\times n}$ we denote its transpose by $A^T$, its inverse by $A^{-1}$ and its transposed inverse by $A^{-T}$.

    \item We write $\nabla_h := (\partial_{x_1},  \partial_{x_2}, \frac{1}{h}\partial_{x_3})^T$ and $\nabla_xg(x,z)$ to indicate that the gradient is taken with respect to $x$.

    \item If a constant $C$ is written as $C(A)$ it is to explicitily highlight its dependence on $A$.

    \item We write $\min\{a,b\} := a \wedge b$ and $\max\{a,b\} := a \vee b$.

    \item Strong convergence is denoted by $\to$ and weak convergence is denoted by $\rightharpoonup$.

    \item By $\mathcal{L}^n(E)$ we denote the n-dimensional Lebesgue measure of some $E \subseteq \mathbb{R}^n$.
    
    \item For the map defining the shallow shell $\Theta_h \circ z_h$ on $\Omega \subseteq \mathbb{R}^3$ later defined in Definition \ref{d:shallowshell}, we use upper case variables $X$, $Y$ and $Z$ to denote elements of $\Theta_h(z_h(\Omega))$ to emphasize the difference in domain compared to $x \in \Omega$.

    \item Elements of the deformed configuration $\Omega^u$, as defined in \eqref{2.1} are denoted by $\xi$.
    
\end{itemize}

 \subsection*{Deformed configuration and topological degree}

 As the deformations $v$ are not homeomorphisms, their image is not necessarily open. Also, deformations are assumed to be Sobolev maps and are therefore defined almost everywhere. This possibly leads to ambigiuities of functions defined on the image of the deformation $v$. To solve these issues the following notion of a deformed configuration is used.

\begin{definition}
    
 Let $\Omega \subseteq \mathbb{R}^n$ be open and bounded and let $v \in  C^0(\overline{\Omega};\mathbb{R}^n)$. The deformed configuration $\Omega^v$, relative to the map $v$, is defined as
 \begin{equation}
     \Omega^v := v(\Omega)\backslash v(\partial \Omega).
     \label{2.1}
 \end{equation}
 \end{definition}

\begin{lemma}
\label{p:opendef}
    Let $\Omega \subseteq \mathbb{R}^n$ be open, bounded , and connected. Let $v \in  C^0(\overline{\Omega};\mathbb{R}^n)$ be almost everywhere injective. Then $\Omega^v$ is an open set.
\end{lemma}
The proof can be found in \cite[Lemma 2.1]{BrescianiIncompressible}.

We further recall that, given $v \in  C^0(\overline{\Omega};\mathbb{R}^n)$, the topological degree is a map $\text{deg}(v,\Omega, .): \mathbb{R}^n\backslash v(\partial \Omega) \to \mathbb{Z}$. It can be axiomatically defined through the following four properties (cf. \cite{degreefonseca}):

\begin{enumerate}[label=(D.\arabic*)]
\setcounter{enumi}{0}
     \item \label{D1} (\emph{Normalization}) $\text{deg}(id_{|\Omega},\Omega, \xi) = 1$ for every $\xi \in \Omega$;
     \item \label{D2} (\emph{Additivity}] if $A_1, A_2 \subseteq \Omega$ are open with $A_1 \cap A_2 = \emptyset$, then\\
     $\text{deg}(v, A_1 \cup A_2, \xi) = \text{deg}(v, A_1, \xi) + \text{deg}(v, A_2, \xi)$
     for every $\xi \in \mathbb{R}^n\backslash v(\partial A_1 \cup \partial A_2)$;
     \item \label{D3} (\emph{Homotopy invariance}) if $H \in C^0([0,1]\times \overline{\Omega}; \mathbb{R}^n)$ and $\gamma: [0,1] \to \mathbb{R}^n$ satisfies\\
     $\gamma(t) \notin H(\{t\}\times \partial \Omega)$ for every $0 \leq t \leq 1$, then\\
     $\text{deg}(H(0,.), \Omega, \gamma(0)) = \text{deg}(H(t,.), \Omega, \gamma(t))$
     for every $0 \leq t \leq 1$.
     \item \label{D4} (\emph{Solvability}) if $\text{deg}(v,\Omega,\xi) \neq 0$ for some $\xi \in \mathbb{R}^n\backslash v(\partial \Omega)$, then there exists $x \in  \Omega$
     such that $\xi = v(x)$.
 \end{enumerate}

A further property deduced from the previous four, which we will use, is as follows.

\begin{enumerate}[label=(D.\arabic*)]
\setcounter{enumi}{4}
     \item \label{D5} (\emph{Stability})  if $\Tilde{v} \in C^0(\overline{\Omega};\mathbb{R}^n)$ satisfies $\norm{\Tilde{v} - v}_{C^0(\overline{\Omega};\mathbb{R}^n)} < \text{dist}(\xi, v(\partial \Omega))$ for some
     $\xi \in \mathbb{R}^n\backslash v(\partial \Omega)$, then $\xi \in \mathbb{R}^n\backslash \Tilde{v}(\partial \Omega)$ and $\text{deg}(\Tilde{v}, \Omega, \xi) = \text{deg}(v, \Omega, \xi)$.
     \end{enumerate}

\subsection*{Rigidity estimate}

A fundamental tool in our analysis is the following well known rigidity estimate (cf. \cite{pRigidity}).

\begin{theorem}
    (Rigidity estimate)
    \label{Thm 2.1}\\
    Let $\Omega \subseteq \mathbb{R}^n$ be a bounded Lipschitz domain, where $n\geq 2$ and $1 < p < \infty$. Then, there exists a constant $C(\Omega, n, p) > 0$, such that for all $v \in W^{1,p}(\Omega;\mathbb{R}^n)$, there exists $R \in SO(n)$, such that
\begin{equation*}
    \norm{\nabla v - R}_{L^p(\Omega;\mathbb{R}^n)}^p \leq C(\Omega, n, p)\int_{\Omega}\emph{dist}^p(\nabla v , SO(n)) dx.
\end{equation*}
In particular, $C$ is invariant under dilations of $\Omega$, and it is also uniform for the uniform bilipschitz images of a unit ball in $\mathbb{R}^n$.
\end{theorem}

\subsection*{Shallow shells}

For $ \omega \subseteq \mathbb{R}^2$ bounded, connected Lipschitz domain we denote the corresponding plate with thickness \textcolor{Green}{1 by $\Omega = \omega \times (-\frac{1}{2},\frac{1}{2})$.} Call $\Omega_h = z_h(\Omega) := \omega \times (-\frac{h}{2},\frac{h}{2})$ the plate with thickness $h$. \textcolor{Green}{Given a function $\theta \in C^\infty(\overline{\omega})$}, we consider the map \textcolor{Green}{$\Theta_h: \overline{\Omega}_h \to \mathbb{R}^3$}defined as 
\begin{equation}
    \Theta_h\circ z_h (x) := (x', h\theta(x')) + hx_3n_h(x'),
    \label{2 shell map}
\end{equation}
where $n_h$ is a unit normal vector field relative to the mid-surface $\hat{\omega}_h:=\Theta_h \circ z_h (\Bar{\omega})$, namely, the graph of the function $h\theta(\cdot)$.

\textcolor{Green}{For comparison to the two main reference note the following differences in notation. In \cite{BrescianiIncompressible} we have $\omega \equiv S$. In \cite{Velcic} we have $z_h \equiv P^h$ and consider the case $f^h \equiv h$.}

\begin{definition}[Shallow shell]
\label{d:shallowshell}
With the above notation, for every $h >0$, the \emph{shallow shell} $\hat{\Omega}_h$ with thickness $h$ is defined as 
\[
\hat{\Omega}_h := \Theta_h \circ z_h(\Omega).
\]
\end{definition}

Under the previous assumptions $\Theta_h$ satisfies the following properties:

\begin{proposition}
    \cite[Theorem 1]{Velcic}\\
    \label{Prop1}
    Let $\Omega, \Theta_h$ and $z_h$ be as defined before. Then, there exists $h_0 = h_0(\theta) > 0$ such that for all $h \leq h_0$ the map $\Theta_h: \Omega_h \to \hat{\Omega}_h$ is a $C^1$-diffeomorphism. Furthermore, the Jacobian matrix $(\nabla \Theta_h) (z_h(x))$ is invertible for all $x \in \overline{\Omega}$ and satisfies
    \begin{align}
        &(\det \nabla\Theta_h)(z_h(x)) = 1 + h^2\delta_h(z_h(x)) ,
        \label{4.6}\\
        &(\nabla \Theta_h)(z_h(x)) = Id - hA(x') + h^2(F_h(z_h(x)),
        \label{4.7}\\
        &(\nabla \Theta_h)^{-1} (z_h(x)) = Id + hA(x') + h^2(G_h(z_h(x)),
        \label{4.8}
    \end{align}
    with
    \begin{align*}
        &\norm{\nabla(\Theta_h)\circ z_h - Id}_{L^{\infty}(\Omega;\mathbb{R}^{3\times 3})} \leq hC,\\
        &\norm{\nabla(\Theta_h)^{-1}\circ z_h - Id}_{L^{\infty}(\Omega;\mathbb{R}^{3\times 3})} \leq hC,
    \end{align*}
    where $A(x') = \begin{pmatrix}
        0 & 0 & \partial_1\theta(x')\\
        0 & 0 & \partial_2\theta(x')\\
        -\partial_1\theta(x') & -\partial_2\theta(x') & 0
    \end{pmatrix}$,
    
    and $\delta_h: \overline{\Omega}_h \to \mathbb{R}$, $F_h: \overline{\Omega}_h \to \mathbb{R}^{3\times 3}$ and $G_h: \overline{\Omega}_h \to \mathbb{R}^{3\times 3}$ satisfy
    \begin{align*}
        &\underset{0<h\leq h_0}{\sup} \ \underset{x_h \in \overline{\Omega}_h}{\max}|\delta_h(z_h(x))| \leq C_0,\\
         &\underset{0<h\leq h_0}{\sup} \ \underset{i,j}{\max} \ \underset{x_h \in \overline{\Omega}_h}{\max}|(F_h(z_h(x)))_{i,j}| \leq C_0,\\
         &\underset{0<h\leq h_0}{\sup} \ \underset{i,j}{\max} \ \underset{x_h \in \overline{\Omega}_h}{\max}|(G_h(z_h(x)))_{i,j}| \leq C_0,
    \end{align*}
    for some constant $C_0 > 0$.
\end{proposition}

%%%%%%%%%%%%%%%%%%%%%%%%%%%%%%%%%%%%%%%%%%%%%%%%%%%%%%%%%%%%%%%%

\section{Setting of the problem and main results}
\label{s:3}
\subsection*{Setting}
Consider the shallow shell with thickness $h$ as defined in \eqref{2 shell map}. We denote by $\pi_h \colon \hat{\Omega}_h \to \hat{\omega}_h$ the projection onto the midsurface along the normal $n_h$.

We will consider deformations $u \in W^{1,p}(\hat{\Omega}_h;\mathbb{R}^3) $ with $p > 3$. This allows us to specify pointwise the value of $u$ and almost everywhere the value of $\det \nabla u$. The class of admissible deformation is made by those $u \in W^{1,p}(\hat{\Omega}_h;\mathbb{R}^3) $ such that $\det \nabla u >0$ pointwise almost everywhere and $u$ is almost everywhere injective.

The magnetizations are maps $m \in W^{1,2}(\hat{\Omega}_h^u;\mathbb{S}^2) $, where $\hat{\Omega}_h^u$ is defined as in \eqref{2.1} ( recall that thanks to the fact that $u$ is almost everywhere injective then by Proposition \ref{p:opendef} the set $\hat{\Omega}_h^u$ is open). Additionally, we consider a map $\psi_{m}: \mathbb{R}^3 \to \mathbb{R}$ which is called the stray field potential, being the weak solution to the Maxwell equation:
\begin{equation}
    \Delta \psi_{m} = \nabla \cdot (\chi_{\hat{\Omega}^u_h}m) \ \ \text{in} \ \mathbb{R}^3,
    \label{3.2}
\end{equation}
where $\chi_{\hat{\Omega}^u_h}m$ extends $m$ by zero to the whole space.

With this, for every $h > 0$ the energy of the magnetoelastic shallow shell $\hat{\Omega}_h$ is given for every $u \in W^{1,p}(\hat{\Omega}_h;\mathbb{R}^3) $ and $m \in W^{1,2}(\hat{\Omega}_h^u;\mathbb{S}^2) $ by
\begin{equation*}
    \mathcal{E}_h(u,m) := \frac{1}{h^\beta}\int_{\hat{\Omega}_h}W^{inc}(\nabla u, m \circ u) \ dx + \alpha\int_{\hat{\Omega}^u_h}|\nabla m|^2 \  d\xi + \frac{1}{2}\int_{\mathbb{R}^3}|\nabla \psi_m|^2 \ d\xi,
\end{equation*}
where $\beta > 2p$ and $\alpha > 0$ and $W^{inc}: \mathbb{R}^{3\times 3}\times \mathbb{S}^2 \to [0,+\infty]$ is called the incompressible energy density:
\begin{equation}
    W^{inc}(F,\nu):= \left\{
    \begin{aligned}
        W(F,\nu), \ \ &\det F = 1\\
        +\infty, \ \ &\det F \neq 1
    \end{aligned}
    \right.,
    \label{3.4}
\end{equation}
 where $W: \mathbb{R}^{3\times 3}\times \mathbb{S}^2 \to [0,+\infty]$ is a continuous and nonlinear elastic energy density, which fulfills the following properties:

\begin{enumerate}[label=(W.\arabic*)]
\setcounter{enumi}{0}
    \item \label{W1} (\emph{Frame indifference}) \begin{equation}
        W(RF, R\nu) = W(F,\nu),
        \label{3.5}
    \end{equation}
    for every $R \in SO(3), F \in \mathbb{R}^{3\times 3}, \nu \in \mathbb{S}^2$,
    \item \label{W2} (\emph{Normalization}) \begin{equation}
        W(Id, \nu) = 0,
        \label{3.6}
    \end{equation}
    for every $\nu \in \mathbb{S}^2$,
    \item \label{W3} (\emph{Growth}) There exists a $C > 0$ such that
    \begin{equation}
        W(F,\nu) \geq C \text{dist}^2(F,SO(3))\vee \text{dist}^p(F,SO(3)),\\
     \label{3.7}
    \end{equation} 
    for every $F \in \mathbb{R}^{3\times 3}, \ \nu \in \mathbb{S}^2$,
    \item \label{W4} (\emph{Regularity}) There exists a $\delta > 0$, such that for all $\nu \in \mathbb{S}^2$ the function
    \begin{equation}
      W(\cdot,\nu) \ \text{is of class} \ C^2,
     \label{3.8}
    \end{equation}
    on the set $\{F \in \mathbb{R}^{3\times 3} : \text{dist}(F,SO(3)) < \delta\}$.
\end{enumerate}

Frame indifference corresponds to the modelling assumption that if the observer is rotated, the energy is not affected. The normalization hypothesis incorporates the fact that the reference configuration $\hat{\Omega}_h$ is a natural state of the material with minimal energy. 

By \ref{W2} and \ref{W4}, we have the following second-order Taylor expansion
\begin{equation}
    W(Id + G, \nu) = \frac{1}{2}Q_3(G,\nu) + \omega(G,\nu),
    \label{3.9}
\end{equation}
for $G \in \mathbb{R}^{3\times 3}$ with $\abs{G}<\delta$ and for every $\nu \in \mathbb{S}^2$, where
\begin{equation*}
    Q_3(G,\nu) := \mathbb{C}^\nu G:G,
\end{equation*}

with the fourth-order tensor $\mathbb{C}^\nu \in \mathbb{R}^{3\times 3 \times 3 \times 3} $ defined by
\begin{equation*}
    \mathbb{C}^\nu := \partial_F^2W(Id, \nu),
\end{equation*}

and, for every $\nu \in \mathbb{S}^2 $, it holds that $\omega(G,\nu) = o(\abs{G}^2)$, as $\abs{G} \to 0^+$.

By \ref{W2}, for every $\nu \in \mathbb{S}^2$, $\mathbb{C}^\nu $ is positive definite, making the quadratic form $Q_3(.,\nu)$ convex.

In addition we assume
\begin{enumerate}[label=(W.\arabic*)]
\setcounter{enumi}{4}
\item The map $\nu \mapsto \mathbb{C}^\nu$ is continuous from $\mathbb{S}^2 \to \mathbb{R}^{3 \times 3 \times 3 \times 3}$, and \label{W5}
    \begin{equation*}
    \overline{\omega}(t) := \sup \big\{ \frac{\omega(G,\nu)}{\abs{G}^2} \ : \ G \in \mathbb{R}^{3 \times 3}, \abs{G} \leq t, \nu \in \mathbb{S}^2 \big\} = o(t^2), \ \text{as} \ t \to 0^+.
    \end{equation*}
\end{enumerate}

\subsection*{Fixing the domain}
As customary in dimension reduction problems, we perform a change of variables to rewrite $\mathcal{E}_h(u,m)$ as an integral functional on the fixed domain $\Omega$ and a fixed deformed domain which is the counterpart to $\hat{\Omega}_h^u$. Using the fact that for for $h > 0$ small, $\Theta_h \circ z_h$ is a $C^1-$Diffeomorphism, we first write the elastic energy functional in terms of the map $u\circ \Theta_h \circ z_h \in W^{1,p}(\Omega;\mathbb{R}^3)$:
\begin{align*}
    \mathcal{E}_h(u,m) &= \frac{1}{h^\beta}\int_{\Theta_h(z_h(\Omega))}W^{inc}(\nabla u, m \circ u) \ dx + \alpha\int_{\hat{\Omega}^u_h}|\nabla m|^2 \  d\xi + \frac{1}{2}\int_{\mathbb{R}^3}|\nabla \psi_m|^2 \ d\xi\\
    &= \frac{1}{h^\beta}\int_{\Omega}W^{inc}((\nabla u) \circ \Theta_h \circ z_h, m \circ u \circ \Theta_h \circ z_h)\abs{\det \nabla(\Theta_h \circ z_h)} \ dx\\
    & \ + \alpha\int_{\hat{\Omega}^u_h}|\nabla m|^2 \  d\xi + \frac{1}{2}\int_{\mathbb{R}^3}|\nabla \psi_m|^2 \ d\xi.
\end{align*}

By the chain rule, there holds
\begin{align*}
    \abs{\det \nabla(\Theta_h \circ z_h) (x)} = \abs{\det((\nabla \Theta_h) \circ z_h(x))\det(\nabla z_h(x))} = h\abs{\det((\nabla \Theta_h) \circ z_h(x))}.
\end{align*}

We therefore obtain
\begin{align*}
    h^{-1}\mathcal{E}_h(u,m) &= \frac{1}{h^\beta}\int_{\Omega}W((\nabla u) \circ \Theta_h \circ z_h, m \circ u \circ \Theta_h \circ z_h) \abs{\det((\nabla \Theta_h) \circ z_h)} \ dx\\
    & \ + \frac{\alpha}{h}\int_{\hat{\Omega}^u_h}|\nabla m|^2 \  d\xi + \frac{1}{2h}\int_{\mathbb{R}^3}|\nabla \psi_m|^2 \ d\xi.
\end{align*}

Notice, that again by the chain rule the following identity holds true:
\begin{equation*}
    (\nabla u)\circ \Theta_h \circ z_h = \nabla_h(u\circ \Theta_h\circ z_h)((\nabla\Theta_h)^{-1} \circ z_h).
\end{equation*}

We set
\begin{align}
    &\kappa_h := \abs{\det((\nabla \Theta_h) \circ z_h(x))},
    \label{3.15}\\
    &M_h := ((\nabla\Theta_h)^{-1} \circ z_h),
    \label{3.16}\\
    &y := u \circ \Theta_h \circ z_{h|_\Omega},
    \label{3.17}
\end{align}

for every $h > 0$. With this he rescaled magnetoelastic energy functional $E_h = h^{-1}\mathcal{E}_h$ is given by

\begin{equation*}
    E_h(y, m) := \frac{1}{h^{\beta}}\int_{\Omega} W^{inc}(\nabla_h y M_h, m \circ y) \ \kappa_h \ dx + \frac{\alpha}{h}\int_{\Omega^y}|\nabla m|^2 \  d\xi + \frac{1}{2h}\int_{\mathbb{R}^3}|\nabla \psi_m|^2 \ d\xi,
\end{equation*}

for every pair of admissible states $(y,m)$.

 \begin{definition}[Admissible states]
We call a deformation $y$ admissible if it belongs to the set
\begin{equation*}
   \mathcal{Y} := \big\{ y \in W^{1,p}(\Omega,\mathbb{R}^3) \ : \ \det \nabla y > 0 \ \text{a.e. in} \ \Omega, \  y \ \text{a.e. injective in} \ \Omega\big\}.
\end{equation*}
The associated admissible magnetizations are maps $m \in W^{1,2}(\Omega^y;\mathbb{S})$, with $\Omega^y$ as in \eqref{2.1}.
Together this defines the set of admissible states $\mathcal{Q}$:
\begin{equation*}
    \mathcal{Q} := \big\{ (y,m) \ | \ y \in \mathcal{Y}, \ m \in W^{1,2}(\Omega^y;\mathbb{S}) \big\}.
\end{equation*}
\end{definition}

We split the energy into three parts: the elastic energy $E^{el}_h$, the exchange energy $E^{exc}_h$ and the magnetostatic energy $E^{mag}_h$, given by
\begin{align*}
    &E^{el}_h := \frac{1}{h^{\beta}}\int_{\Omega} W^{inc}(\nabla_h y M_h, m \circ y) \ \kappa_h \ dx\\
    &E^{exc}_h := \frac{\alpha}{h}\int_{\Omega^y}|\nabla m|^2 \  d\xi\\
    &E^{mag}_h := \frac{1}{2h}\int_{\mathbb{R}^3}|\nabla \psi_m|^2 \ d\xi.
\end{align*}

 With the following Lemma, we collect some properties of the quantities introduced above relative to the geometry of the shallow shell that are uselful in upcoming proofs.

\begin{lemma}
\label{L3.1}
    For the maps $\kappa_h$ and $M_h$ defined in \eqref{3.15} and \eqref{3.16} and $h > 0$ sufficiently small, there holds
    \begin{align}
        &\kappa_h \leq C_1,
        \label{4.38}\\
        &\kappa_h \geq C_2,
        \label{4.39}\\
        &\norm{M_h}_{L^\infty(\Omega;\mathbb{R}^{3\times 3})} \leq C_1,
        \label{4.40}\\
        &\norm{M_h}_{L^\infty(\Omega;\mathbb{R}^{3\times 3})} \geq C_2,
        \label{4.41}\\
        &\norm{M_h \xi}_{L^2(\mathbb{R}^3;\mathbb{R}^{3\times 3})} \geq C \ \norm{\xi}_{L^2(\mathbb{R}^3;\mathbb{R}^{3\times 3})},
        \label{4.41b}
    \end{align}
    for all $\xi \in \mathbb{R}^3$, where $C_1, C_2, C > 0$. Moreover, as $h \to 0^+$ for all $ 1 \leq r \leq \infty$, there holds

    \begin{align}
        &\kappa_h \to 1 \ \ \text{in} \ L^r(\Omega),
        \label{4.42}\\
        &M_h \to Id \ \ \text{in} \ L^r(\Omega; \mathbb{R}^{3 \times 3})
        \label{4.43}
    \end{align}
\end{lemma}

\begin{proof}

    This Lemma is a direct consequence of Proposition \ref{Prop1}. By \eqref{4.6}, it follows that
    \begin{align*}
        \kappa_h = \abs{1 + h^2\delta_h \circ z_h},
    \end{align*}
    and also
    \begin{align*}
        \kappa_h &\leq 1 + h^2C_0,\\
        \kappa_h &\geq 1 - h^2C_0.
    \end{align*}
    Inequalities \eqref{4.38} and \eqref{4.39} follow from this observation if $h > 0$ is sufficiently small. Property \eqref{4.40} holds because of \eqref{4.8}. By the inverse triangle inequality, there holds
    \begin{align*}
        \norm{M_h}_{L^\infty(\Omega;\mathbb{R}^{3\times 3})} &= \norm{Id + h(A \circ \pi) + h^2(G_h \circ z_h)}_{L^\infty(\Omega;\mathbb{R}^{3\times 3})}\\
        &\geq \norm{Id + h(A \circ \pi)}_{L^\infty(\Omega;\mathbb{R}^{3\times 3})} - h^2 \norm{G_h \circ z_h}_{L^\infty(\Omega;\mathbb{R}^{3\times 3})}\\
        &\geq  \norm{Id}_{L^\infty(\Omega;\mathbb{R}^{3\times 3})} - h  \norm{A \circ \pi}_{L^\infty(\Omega;\mathbb{R}^{3\times 3})} - h^2 \norm{G_h \circ z_h}_{L^\infty(\Omega;\mathbb{R}^{3\times 3})}\\
        &\geq 1 - hC - h^2C_0.
    \end{align*}
    For $h$ small enough \eqref{4.41} follows. By an analogous arguement \eqref{4.41b} is shown. 
    
    Eventually, \eqref{4.42} and \eqref{4.43} follow from \eqref{4.6} and \eqref{4.8}.
    
\end{proof}

\subsection*{Main results}

First we introduce the limiting functional. Analogously to \cite{BrescianiIncompressible}, for every $H \in \mathbb{R}^{2\times 2} $ and $\nu \in \mathbb{S}^2$ we set
\begin{equation*}
    \begin{aligned}
         Q_2^{inc}(H, \nu) := \min \Bigg\{&Q_3\bigg(\left(\begin{array}{c|c}
    	H & 0'\\ 
    	\hline 
    	(0')^T & 0
    \end{array}\right) 
    + c \otimes e_3 + e_3 \otimes c, \nu \bigg) : c \in \mathbb{R}^3, \\
         &\text{tr}\bigg(\left(\begin{array}{c|c}
    	H & 0'\\ 
    	\hline 
    	(0')^T & 0
    \end{array}\right) 
    + c \otimes e_3 + e_3 \otimes c\bigg) = 0\Bigg\}.
    \end{aligned}
\end{equation*}

With this, for every triple $(u,v,\lambda) \in W^{1,2}(\omega;\mathbb{R}^2) \times W^{2,2}(\omega) \times W^{1,2}(\omega;\mathbb{S}^2) $ we define the limiting energy as
\begin{equation*}
    \begin{aligned}
        E(u,v,\lambda) := \ &\frac{1}{2}\int_\omega Q_2^{inc}(\text{sym}(\nabla ' u + \nabla'v \otimes \nabla'\theta), \lambda) \ dx'\\
        + \ &\frac{1}{24}\int_\omega Q_2^{inc}((\nabla ')^2v, \lambda) \ dx'
        + \alpha\int_\omega \abs{\nabla ' \lambda}^2 \ dx' + \frac{1}{2}\int_\omega \abs{(\lambda)^3}^2 \ dx'.
    \end{aligned}
\end{equation*}

We again split the above energy functional into three parts as follows
\begin{align*}
    &E^{el}(u,v,\lambda) = \frac{1}{2}\int_\omega Q_2^{inc}(\text{sym}(\nabla ' u + \nabla'v \otimes \nabla'\theta), \lambda) \ dx' + \frac{1}{24}\int_\omega Q_2^{inc}((\nabla ')^2v, \lambda) \ dx'\\
    &E^{exc}(\lambda) = \alpha\int_\omega \abs{\nabla ' \lambda}^2 \ dx'\\
    &E^{mag}(\lambda) = \frac{1}{2}\int_\omega \abs{(\lambda)^3}^2 \ dx'.
\end{align*}

Lastly, we introduce in-plane and out-of-plane displacements.

\begin{definition}
(In-plane and out-of-plane displacement)
\label{Def3.1}\\
    Let $y_h \in W^{1,p}(\Omega;\mathbb{R}^3)$ be a deformation. The associated in-plane displacement $U_h^{y_h}: \omega \to \mathbb{R}^2 $ and out-of-plane displacement $V_h^{y_h}: \omega \to \mathbb{R}$ are given by
    \begin{align}
        &U_h^{y_h}(x') := \frac{1}{h^{\frac{\beta}{2}}}\int_{-\frac{1}{2}}^{\frac{1}{2}}\big(y_h'(x',x_3) - x'\big) \ dx_3,
        \label{3.23}\\
        &V_h^{y_h}(x') := \frac{1}{h^{\frac{\beta}{2}-1}}\int_{-\frac{1}{2}}^{\frac{1}{2}}\big((y_h)^3(x',x_3) - h\theta(x')\big) \ dx_3.
        \label{3.24}
    \end{align}
\end{definition}

We are now in position to state the two main contributions of this paper.
\begin{theorem}
\label{thm1}
    (Compactness and lower bound)\\
    Assume $p > 3$ and $\beta > 2p$ and that the elastic energy density $W$ satisfies \ref{W1}--\ref{W5}. Let $((y_h,m_h))_h \subseteq \mathcal{Q}$ be a sequence of admissible states which satisfies $E_h(y_h,m_h) \leq C$ for every $h>0$. Then, there exists a sequence of rotations $(Q_h)_h \subseteq SO(3)$ and a sequence of translation vectors $(c_h)_h \subseteq \mathbb{R}^3$, such that by setting $\Tilde{y}_h := T_h \circ y_h$ and $\Tilde{m}_h := Q_h^Tm_h \circ T_h^{-1}$, where $T_h: \mathbb{R}^3 \to \mathbb{R}^3$ is the rigid motion defined by $T_h(\xi) := Q_h^T\xi - c_h$ for every $\xi \in \mathbb{R}^3$, we have, as $h \to 0^+$:
    \begin{align*}
        &U_h^{\Tilde{y}_h} \rightharpoonup u \ \ \text{in} \ W^{1,2}(\omega;\mathbb{R}^3) \ \ \text{for some} \ u \in W^{1,2}(\omega;\mathbb{R}^3);\\
         &V_h^{\Tilde{y}_h} \to v \ \ \text{in} \ W^{1,2}(\omega) \ \ \text{for some} \ v \in W^{2,2}(\omega);\\
         &\Tilde{m}_h \circ \Tilde{y}_h \to \lambda \ \ \text{in} \ L^r(\Omega;\mathbb{R}^3) \ \ \text{for every} \ 1 \leq r < \infty, \ \text{for some} \ \lambda \in W^{1,2}(\omega;\mathbb{S}^2).
    \end{align*}
    Moreover, the following inequality holds:
    \begin{equation*}
        E(u,v,\lambda) \leq \liminf_{h\to 0^+}E_h(y_h,m_h).
    \end{equation*}
\end{theorem}

\begin{theorem}
\label{thm2}
    (Recovery sequence)\\
    Assume $p > 3$ and $\beta > 2p$ and that the elastic energy density $W$ satisfies \ref{W1}--\ref{W5}. Then, for every $(u,v,\lambda) \in W^{1,2}(\omega;\mathbb{R}^2) \times W^{2,2}(\omega) \times W^{1,2}(\omega;\mathbb{S}^2)$, there exists a sequence of admissible states $((y_h,m_h))_h \subseteq \mathcal{Q}$ such that, as $h \to 0^+$:
    \begin{align*}
        &u_h := U_h^{y_h} \to u \ \ \text{in} \ W^{1,2}(\omega;\mathbb{R}^3);\\
         &v_h := V_h^{y_h} \to v \ \ \text{in} \ W^{1,2}(\omega));\\
         &m_h \circ y_h \to \lambda \ \ \text{in} \ L^r(\Omega;\mathbb{R}^3) \ \ \text{for every} \ 1 \leq r < \infty.
    \end{align*}
    Moreover, the following equality holds:
    \begin{equation*}
        \lim_{h\to 0^+}E_h(y_h,m_h) = E(u,v,\lambda).
    \end{equation*}
\end{theorem}
Analogously to \cite[Section 3.4]{BrescianiIncompressible}, we can state the two main results from Theorem \ref{thm1} and \ref{thm2} in terms of $\Gamma$-convergence. For this, we first introduce the functional $\mathcal{M}_h(y,m): \mathbb{R}^3 \to \mathbb{R}^3$ by setting
\begin{equation*}
    \mathcal{M}_h(y,m) := (\chi_{\Omega^{y}}m)\circ\Theta_h\circ z_h,
\end{equation*}
for $h > 0$ and admissible states $(y,m) \in \mathcal{Q}$.

Recalling \eqref{3.23} and \eqref{3.24}, we introduce the functionals
\begin{equation*}
    \mathcal{I}_h, \mathcal{I}: W^{1,p}(\Omega;\mathbb{R}^3) \times W^{1,2}(\omega;\mathbb{R}^2)\times W^{1,2}(\omega)\times L^2(\mathbb{R}^3;\mathbb{R}^3) \to \mathbb{R}^3\cup \{+\infty\},
\end{equation*}
given by
\begin{equation*}
    \mathcal{I}_h(y,u,v,\mu) := \begin{cases} 
      E_h(y,m) & \text{if} \ u = U_h^{y}, v = V_h^y \ \text{and} \ \mu = \mathcal{M}_h(y,m)\\
       & \text{for some} \ m \in W^{1,2}(\Omega^y;\mathbb{S}^2), \\
      +\infty & \text{otherwise}
   \end{cases}
\end{equation*}
and
\begin{equation*}
    \mathcal{I}(y,u,v,\mu) := \begin{cases} 
      E(u,v,\lambda) & \text{if} \ y = \Theta_0 \circ z_0, v \in W^{2,2}(\omega) \ \text{and} \ \mu = \chi_\Omega\lambda\\
       & \text{for some} \ \lambda \in W^{1,2}(\Omega^y;\mathbb{S}^2), \\
      +\infty & \text{otherwise}. 
   \end{cases}
\end{equation*}
Similarly to \cite[Corollary 3.1]{BrescianiIncompressible}, we obtain the following result:

\begin{corollary}
    ($\Gamma$-convergence)
    \label{col1}\\
    The functionals $(\mathcal{I}_h)$ $\Gamma$-converge to $\mathcal{I}$ with respect to the strong product topology, as $h \to 0^+$. Namely, the following holds:
    \begin{itemize}
        \item[(i)] (Liminf inequality) for every sequence $((y_h,u_h,v_h,\mu_h))_h \subseteq W^{1,p}(\Omega;\mathbb{R}^3) \times W^{1,2}(\omega;\mathbb{R}^2)\times W^{1,2}(\omega)\times L^2(\mathbb{R}^3;\mathbb{R}^3)$ such that $(y_h,u_h,v_h,\mu_h) \to (y,u,v,\mu)$ in $W^{1,p}(\Omega;\mathbb{R}^3) \times W^{1,2}(\omega;\mathbb{R}^2)\times W^{1,2}(\omega)\times L^2(\mathbb{R}^3;\mathbb{R}^3)$ for some $(y,u,v,\mu) \in W^{1,p}(\Omega;\mathbb{R}^3) \times W^{1,2}(\omega;\mathbb{R}^2)\times W^{1,2}(\omega)\times L^2(\mathbb{R}^3;\mathbb{R}^3)$, we have
        \begin{equation*}
            \mathcal{I}(y,u,v,\mu) \leq \underset{h \to 0^+}{\liminf} \ \mathcal{I}_h(y_h,u_h,v_h,\mu_h),
        \end{equation*}
        \item[(ii)] (Recovery sequence) for every $(y, u, v, \mu) \in W^{1,p}(\Omega;\mathbb{R}^3) \times W^{1,2}(\omega;\mathbb{R}^2)\times W^{1,2}(\omega)\times L^2(\mathbb{R}^3;\mathbb{R}^3)$ there exists a sequence $((y_h,u_h,v_h,\mu_h))_h \subseteq W^{1,p}(\Omega;\mathbb{R}^3) \times W^{1,2}(\omega;\mathbb{R}^2)\times W^{1,2}(\omega)\times L^2(\mathbb{R}^3;\mathbb{R}^3)$ such that $(y_h,u_h,v_h,\mu_h) \to (y,u,v,\mu)$ in $W^{1,p}(\Omega;\mathbb{R}^3) \times W^{1,2}(\omega;\mathbb{R}^2)\times W^{1,2}(\omega)\times L^2(\mathbb{R}^3;\mathbb{R}^3)$ and we have
        \begin{equation*}
            \mathcal{I}(y, u, v, \mu) = \underset{h \to 0^+}{\lim} \ \mathcal{I}_h(y_h, u_h, v_h, \mu_h).
        \end{equation*}
    \end{itemize}
    The same result holds for the weak product topology instead of the strong one.
\end{corollary}

Corollary \ref{col1} presents a summary of Theorems \ref{thm1} and \ref{thm2} and will be proven at the very end of Section \ref{s:7}. It generally offers less information than the two main results, as the sequence of functionals $(\mathcal{I}_h)$ are not coercive. Consequently, the Fundamental Theorem of $\Gamma$-convergence is not applicable and the convergence of minimizers cannot be deduced.

%%%%%%%%%%%%%%%%%%%%%%%%%%%%%%%%%%%%%%%%%%%%%%%%%%%%%%%%%%%%%%%%%%%%

\section{Approximation by rotations and compactness}
\label{s:4}

\subsection*{Approximation by rotations}

This section is dedicated to the first crucial step for the proof of Theorem \ref{thm1}. The following approximation result for magnetoelastic shallow shells hinges on the rigidity estimate stated in Theorem \ref{Thm 2.1}. Before stating it, we observe that the rotation $R$ appearing in Theorem \ref{Thm 2.1} does not depend on the exponent $p$ (see \cite[Remark 2.3]{BrescianiIncompressible}). This immediately implies that for every $u \in W^{1,p}(\Omega;\mathbb{R}^3)$, we find a rotation $R \in SO(3)$ such that 
\begin{equation}
    \int_\Omega \abs{\nabla u(x) - R}^q dx \leq C(\Omega, q) K(u,\Omega)
    \label{4.18},
\end{equation}
for $q \in \{2,p\}$, whenever we set
\begin{equation*}
    K(u,\Omega) := \int_\Omega \text{dist}^2(\nabla u, SO(3)) \vee \text{dist}^p(\nabla u, SO(3))dx.
\end{equation*}

\textcolor{Green}{
Suppose that our energy sequence is equi-bounded, meaning that
\begin{equation*}
    E_h(y_h, m_h) \leq C, \ \ \ \forall \ h > 0.
\end{equation*}
Then, by the growth condition \ref{W3} there holds
\begin{equation}
   K_h(y_h,\Omega) \leq h^{\beta}E^{el}_h(y_h,m_h) \leq Ch^{\beta}.
   \label{4.19}
\end{equation}
}

In the next theorem we extend the approximation by rotations result obtained in \cite[Lemma 4.1]{BrescianiIncompressible} to the setting of magnetoelastic shallow shells. 

\begin{theorem} (Approximation by rotations) \\
\label{Thm4.1}
    Let $u_h \in W^{1,p}(\hat{\Omega}_h;\mathbb{R}^3) $ and assume that $h^{-p}K(u_h,\hat{\Omega}_h)$ is sufficiently small, for $h \ll 1$.\\
    Then, there exist matrix fields $R_h \in W^{1,p}(\hat{\omega}_h;\mathbb{R}^{3\times 3}) $ such that $R_h(\pi_h(Z)) \in SO(3)$ for all $Z \in \hat{\Omega}_h$, and matrices $Q_h \in SO(3)$ with the following properties:
    \begin{align*}
        &(i) \ \norm{\nabla u_h - R_h\pi_h}_{L^q(\hat{\Omega}_h;\mathbb{R}^{3 \times 3})}^q \leq CK(u_h,\hat{\Omega}_h),\\
        &(ii) \ \norm{ \hat{\nabla} R_h }_{L^q(\hat{\omega}_h;\mathbb{R}^{3 \times 3})}^q \leq Ch^{-q-1}K(u_h,\hat{\Omega}_h),\\
        &(iii) \ \norm{R_h - Q_h}_{L^q(\hat{\omega}_h;\mathbb{R}^{3 \times 3})}^q \leq Ch^{-q-1}K(u_h,\hat{\Omega}_h),
    \end{align*}
    where the constant $C$ is independent of $u_h$ and $h$.
\end{theorem}
\begin{proof}

\textbf{Step 1} (Applying the Rigidity Estimate locally):

Consider a ball $B_h(X) \subseteq \mathbb{R}^3$ with radius $h$ centered at $X \in \hat{\omega}_h$ and set
    \begin{equation*}
        \begin{aligned}
            D_{X,h} := B_h(X) \cap \hat{\omega}_h\\
            B_{X,h} := \pi^{-1}_h(D_{X,h}) \cap \hat{\Omega}_h.
        \end{aligned}
    \end{equation*}
The cylinders $B_{X,h}$ are open and bounded domains with Lipschitz boundaries. This allows us to apply the Rigidity Estimate \eqref{4.18} to find rotations $R_{X,h}$ such that
    \begin{equation}
        \int_{B_{X,h}} \abs{\nabla u_h(y) - R_{X,h}}^q dy \leq C K(u_h,B_{X,h}).
        \label{4.23}
    \end{equation}
By the fact that $\Theta_h \circ z_h$ is a $C^1$-Diffeomorphism, we know that $\hat{\Omega}_h$ is bi-Lipschitz equivalent to $\Omega_h$. As a consequence, all cylinders $B_{X,h}$ are bi-Lipschitz equivalent to each other for every sufficiently small $h >0$. Hence we can apply Theorem \ref{Thm 2.1} to infer that the constant $C$ in the previous estimate can be chosen independently from $h$ and $x$.

    \textbf{Step 2} (Cut-off function):

Take a function $\rho \in \mathcal{C}_c^\infty(\mathbb{R};[0,1])$ with supp$(\rho) \subseteq (-\frac{1}{2},\frac{1}{2})$, and $\int_{\mathbb{R}}\rho = 1$ which is equal to 1 in a small neighbourhood of 0. For $X \in \hat{\omega}_h$ define the function $\eta_X: \hat{\omega}_h \to \mathbb{R}$ by setting
    \begin{equation*}
        \eta_X(Z) = \frac{\rho(\abs{\pi_h Z - X}/h)}{\int_{\hat{\Omega}_h}\rho(\abs{\pi_h Y - X}/h) \ dY}.
    \end{equation*}
    
   Then, $\eta_X$ satisfies the following:
    \begin{align}
        &\text{For} \ Z \notin B_{X,h}, \ \ \eta_X(Z) = 0,
        \label{4.25}\\
        &\int_{\hat{\Omega}_h}\eta_X(Z) \ dZ = 1,
        \label{4.26}\\
        &\norm{\eta_X}_{L^\infty(\hat{\Omega}_h)} \leq Ch^{-3},
        \label{4.27}\\
        &\norm{\nabla\eta_X}_{L^\infty(\hat{\Omega}_h;\mathbb{R}^3)} \leq Ch^{-4}.
        \label{4.28}
    \end{align}
Property \eqref{4.25} follows from the fact that $Z \notin B_{X,h}$ implies $\abs{\pi_h Z - X}/h > 1$ and consequently there holds $\rho(\abs{\pi_h Z - X}/h) = 0$. Condition \eqref{4.27} can be shown by using $\norm{\rho}_{L^\infty(\mathbb{R})} = 1$ and applying the change of variables $\Tilde{Y} = (Y - X)/h$. The argument for \eqref{4.28} is analogous with the addition of using $\norm{\nabla_x\rho}_{L^\infty(\mathbb{R})} \leq Ch^{-1}$.

    \textbf{Step 3} (Global Estimate):

Consider the matrix field $\Tilde{R} \in W^{1,p}(\hat{\omega}_h;\mathbb{R}^{3\times 3}) $ given by
    \begin{equation*}
        \Tilde{R}(X) = \int_{\hat{\Omega}_h}\eta_X(Z)\nabla u(Z) \ dZ.
    \end{equation*}

Using \eqref{4.25}, Hölder inequality together with \eqref{4.26} and the local rigidity estimate \eqref{4.23}, we obtain:
    \begin{equation}
        \begin{aligned}
            \abs{\Tilde{R}(X) - R_{X,h}}^q &= \abs{\int_{B_{X,h}}\eta_X(Z)\nabla u(Z) \ dz - R_{X,h}\int_{B_{X,h}}\eta_X(Z) \ dZ }^q\\
        &= \abs{\int_{B_{X,h}}\eta_X(Z)\big(\nabla u(Z) - R_{X,h}\big) \ dZ}^q\\
        &\leq \norm{\eta_X}_{L^\infty(\hat{\Omega}_h)}^q\mathcal{L}^3(B_{X,h})^{q-1} \int_{B_{X,h}}\abs{\nabla u (Z) - R_{X,h}}^q \ dz\\
        &\leq Ch^{-3} K(u_h,B_{X,h}).
        \end{aligned}
        \label{4.30}
    \end{equation}

Similarily, using in addition \eqref{4.28}, we get the following estimate:
    \begin{equation}
        \begin{aligned}
            \abs{\hat{\nabla} \Tilde{R}(X)}^q &= \abs{\nabla_X\bigg(\int_{\hat{\Omega}_h}\eta_X(Z)\nabla u(Z) \ dZ\bigg)}^q\\
        &= \abs{\int_{B_{X,h}}(\nabla_X\eta_X (Z)) \big(\nabla u(Z) - R_{X,h}\big) \ dZ}^q\\
        &\leq \norm{\nabla_X\eta_X}^q_{L^\infty(\hat{\Omega}_h)}\mathcal{L}^3(B_{X,h})^{q-1}\int_{B_{X,h}}\abs{\nabla u(Z) - R_{X,h}}^q \ dZ\\
        &\leq Ch^{-q-3}K(u_h,B_{X,h}).
        \end{aligned}
        \label{4.31}
    \end{equation}
For any $\Tilde{X} \in B_{X,h} \cap \hat{\omega}_h$, using the Fundamental Theorem of Calculus, by setting $\gamma_h(t):= \Theta_h(x' + t(\Tilde{x}' - x'),0)$ for $t \in [0,1]$, we infer
 \begin{align*}
        \abs{\Tilde{R}(\Tilde{X}) - \Tilde{R}(X)} &= \abs{\int_0^1 \frac{d}{dt} \Tilde{R}(\gamma_h(t)) \ dt}\\
        &= \abs{\int_0^1 \nabla\Tilde{R}(\gamma_h(t)) \cdot \dot{\gamma}_h(t) \ dt}.
    \end{align*}

    By the Mean Value Theorem, we find $\tau \in [0,1]$ such that
    \begin{equation*}
        \abs{\int_0^1 \nabla\Tilde{R}(\gamma_h(t)) \cdot \dot{\gamma}_h(t)  \ dt} \leq  \text{length}(\gamma_h) \ \abs{\nabla\Tilde{R}(\gamma_h(\tau))}.
    \end{equation*}

    Since $ \text{length}(\gamma_h) \leq |x'-\tilde{x}'| + h \text{Lip}(\theta) \leq Ch$, for some $C > 0$, writing $Y_h := \gamma_h(\tau)$ and using \eqref{4.31}, we deduce
    \begin{equation}
        \begin{aligned}
            \abs{\Tilde{R}(\Tilde{X}) - \Tilde{R}(X)}^q &\leq C h^q \abs{\hat{\nabla}\Tilde{R}(Y_h)}^q\\
        &\leq C h^{-3}K(u_h, B_{Y_h,h})\\
        &\leq C h^{-3}K(u_h, B_{X,2h}),
        \end{aligned}
        \label{4.32}
    \end{equation}
    
    Combining \eqref{4.23}, \eqref{4.30} and \eqref{4.32} results in
    \begin{equation}
        \begin{aligned}
            &\int_{B_{X,h}}\abs{\nabla u(Z) - \Tilde{R}(\pi_h(Z))}^q \ dZ\\
        &\leq \int_{B_{X,h}}\bigg(\abs{\nabla u(Z) - R_{X,h}} + \abs{\Tilde{R}(X) - R_{X,h}} + \abs{\Tilde{R}(\pi_h(Z)) - \Tilde{R}(X)}\bigg)^q \ dZ\\
        &\leq CK(u_h, B_{X,2h}).
        \end{aligned}
        \label{localrigidity}
    \end{equation}

    \textbf{Step 4} (Covering):

Consider a square $Q = [-a,a]^2 \subseteq \mathbb{R}^3$, for some $a > 0$ such that $\omega \subseteq Q$ as well as the equidistant partition $(\frac{n}{2^i})_{n=0}^{2^i} \subseteq [0,1]$ with distance $\frac{1}{2^i}$ for some $i \in \mathbb{N}$. Using the affine transformation $\varphi: [0,1] \to [-a,a]$, defined as
    \begin{equation*}
            \varphi(x) := -a(1-2x),
    \end{equation*}

     with $\{\varphi(\frac{n}{2^i})\}_n$ we get an equidistant partition of $[-a,a]$ with distance $\frac{2a}{2^i}$. With this partition we can define a grid of equidistant points $(x'_n)_{n=0}^{N_i}$ with distance $\frac{2a}{2^i}$ to their neighbors on $Q$. Define now $D_{r_i}(x'_n)$ which are discs with radius $r_i = \frac{4a}{2^i}$ around $x'_n$. By construction it holds that $\omega \subseteq\bigcup_{n = 0}^{N_i}D_{r_i}(x'_n)$.  Now suppose that $\frac{2a}{2^i} \sim h$, or to be more precise $i := \bigl \lceil \log_2(\frac{2a}{h})\bigr \rceil$. Then, for all points $x' \in \omega$ the covering number of the discs is independent of $h$. As the map $\Theta_h \circ z_h$ is bijective this implies that the covering of $\hat{\omega}_h$ by $\Theta_h\circ z_h (D_{r_i}(x'_n))$ is again independent of $h$. As the map $\pi_h$ is well defined, the independence is preserved if we consider $B_{X_n,h} := \pi_h^{-1}(\Theta_h\circ z_h (D_{r_i}(x'_n))) \cap \hat{\Omega}_h$.
     
     This implies the existence of a uniform constant $C$ which does not depend on $h$ (or $i$) such that by summing up the local rigidity estimates from \eqref{localrigidity}, we get
    \begin{align*}
        \int_{\hat{\Omega}_h}\abs{\nabla u_h - \Tilde{R}\pi_h}^q \ dX \leq C\sum_{n=1}^{N_i}K(u_h,  B_{X_n,2h}).
    \end{align*}

   \textcolor{Green}{By definition we have $\bigcup_{n = 1}^{N_i}B_{X_n,2h} = \hat{\Omega}_h$ and therefore}
    \begin{align*}
        \sum_{n=1}^{N_i}K(u_h, B_{X_n,2h}) = K(u_h,\hat{\Omega}_h),
    \end{align*}
    such that
    \begin{equation}
        \int_{\hat{\Omega}_h}\abs{\nabla u_h - \Tilde{R}\pi_h}^q \leq C K(u_h, \hat{\Omega}_h).
        \label{4.34}
    \end{equation}
    
    Integrating $\abs{\nabla \Tilde{R}}^q$ over $\pi_h(B_{X_n,h})$ and using \eqref{4.31} results in
    \begin{align*}
        \int_{\pi_h(B_{X_n,h})}\abs{\hat{\nabla} \Tilde{R}}^q \ dX &\leq \mathcal{L}^2(\pi_h(B_{X_n,h})) \max_{Y \in \pi_h(B_{X_n,h})}\abs{\hat{\nabla} \Tilde{R}(Y)}^q\\
        &\leq Ch^{2}h^{-q-3}\max_{Y \in \pi_h(B_{X_n,h})}\int_{B_{Y,h}}\text{dist}^p(\nabla u_h, SO(3)) \vee \text{dist}^2(\nabla u_h, SO(3)) \ dX\\
        &\leq Ch^{-q-1}\int_{B_{X_n,2h}}\text{dist}^p(\nabla u_h, SO(3)) \vee \text{dist}^2(\nabla u_h, SO(3)) \ dX\\
        &= Ch^{-q-1}K(u_h,B_{X_n,2h}).
    \end{align*}
    
    Using $\bigcup_{n = 1}^{N_i}\pi_h(B_{X_n,h}) = \hat{\omega}_h $, we deduce (ii).

    \textbf{Step 5} (Showing (i)):

    From \eqref{4.30} it follows that
    \begin{equation*}
        \text{dist}^q(\Tilde{R}(X), SO(3)) \leq Ch^{-3}K(u_h(X), \hat{\Omega}_h).
    \end{equation*}
    \textcolor{Green}{By \eqref{4.19} we know that $h^{-3}K(u, \hat{\Omega}_h)$ is sufficiently small}, which implies that the orthogonal projection of $\Tilde{R}$ onto $SO(3)$
    \begin{equation*}
        R := \Pi_{SO(3)}\Tilde{R}.
    \end{equation*}
     is well defined.

    Since $\Tilde{R} \in W^{1,p}(\hat{\omega}_h;\mathbb{R}^{3\times 3}) $ and $\Pi_{SO(3)}$ is Lipschitz, it follows that $R \in W^{1,p}(\hat{\omega}_h;SO(3))$.

   As a next step, we show
    \begin{equation}
        \int_{\hat{\Omega}_h}\text{dist}^q(\Tilde{R}\circ\pi_h, SO(3)) \ dX \leq \int_{\hat{\Omega}_h}\text{dist}^q(\nabla u_h, SO(3)) \ dX.
        \label{4.35}
    \end{equation}

    Using the fact that $\bigcup_{n = 1}^{N_i}B_{X_n,2h} = \hat{\Omega}_h $ and
    \begin{equation*}
        \abs{\Tilde{R}(\pi_h(X)) - \Pi_{SO(3)}(\Tilde{R}(\pi_h(X)))}^q \leq \abs{\Tilde{R}(\pi_h(X)) - R_{x_i,h}}^q,
    \end{equation*}
    
     we get
    \begin{align*}
        & \ \ \int_{\hat{\Omega}_h}\text{dist}^q(\Tilde{R} (\pi_h(X)), SO(3)) \ dX = \int_{\hat{\Omega}_h}\abs{\Tilde{R}(\pi_h(X)) - \Pi_{SO(3)}\Tilde{R}(\pi_h(X))}^q \ dX\\
        &\leq C \sum_{n = 1}^{N_i} \int_{B_{X_n,h}}\abs{\Tilde{R}(\pi_h(X)) - R_{X_n,h}}^q \ dX\\
        &= C \sum_{n = 1}^{N_i} \int_{B_{X_n,h}}\abs{\int_{\hat{\Omega}_h}\eta_X(Z)\nabla u_h(Z) \ dZ - R_{X_n,h}}^q \ dX\\
        &= C \sum_{n = 1}^{N_i} \int_{B_{X_n,h}}\abs{\int_{\hat{\Omega}_h}(\nabla u_h(Z)  - R_{X_n,h})\eta_X(Z) \ dZ}^q \ dX.
    \end{align*}

    We now use the fact that $\eta_X(Z) = 0$ for $Z \notin B_{X,h}$ and Jensen's inequality for the function $f \mapsto \abs{f}^q$. Note also that since $X \in B_{X_n,h}$ it follows that $B_{X,h} \subseteq B_{X_n,2h}$. Together this yields
    \begin{align*}
        &\sum_{n = 1}^{N_i} \int_{B_{X_n,h}}\abs{\int_{\hat{\Omega}_h}(\nabla u_h(Z)  - R_{X_n,h})\eta_X(Z) \ dZ}^q \ dX\\
        &\leq  \sum_{n = 1}^{N_i} \int_{B_{X_n,h}}\bigg(\int_{B_{X,h}}\abs{\nabla u_h(Z) - R_{X_n,h}}^q \eta_X(Z) \ dZ\bigg) \ dX\\
        &\leq \sum_{n = 1}^{N_i} \int_{B_{X_n,h}}\bigg(\int_{B_{X_n,2h}}\abs{\nabla u_h(Z) - R_{X_n,h}}^q \eta_X(Z) \ dZ\bigg) \ dX.
    \end{align*}

    Using Fubini's theorem together with the fact that
    \begin{align*}
        \abs{\int_{B_{X_n,2h}} \eta_X(Z) \ dX} \leq \norm{\eta_X}_{L^\infty(\hat{\Omega}_h)}\mathcal{L}^3(B_{X_n,2h}) \leq C,
    \end{align*}
   by the Rigidity Estimate from Proposition \eqref{4.18}, we conclude that:
    \begin{align*}
        &C \sum_{n = 1}^{N_i} \int_{B_{X_n,h}}\bigg(\int_{B_{X_n,2h}}\abs{\nabla u_h(Z) - R_{X_n,h}}^q \eta_X(Z) \ dZ\bigg) \ dX\\
        &= C \sum_{n = 1}^{N_i} \int_{B_{X_n,2h}} \abs{\nabla u_h(Z) - R_{X_n,h}}^q \bigg(\int_{B_{X_n,h}} \eta_X(Z) \ dX\bigg) \ dZ\\
        &\leq C \sum_{n = 1}^{N_i} \int_{B_{X_n,2h}}\text{dist}^q(\nabla u(Z)_h, SO(3)) \ dZ\\
        &\leq \Tilde{C}\int_{\hat{\Omega}_h}\text{dist}^q(\nabla u_h(Z), SO(3)) \ dZ,
    \end{align*}
    as desired. The constant $\Tilde{C}$ depends on the maximal covering number of $\hat{\Omega}_h$ by the family $(B_{X_n,2h})_{n = 1}^{N_i}$.

    Using \eqref{4.34} together with \eqref{4.35}, it follows that
    \begin{align*}
        \int_{\hat{\Omega}_h}\abs{\nabla u_h - R\pi_h}^q \ dX &\leq C\bigg(\int_{\hat{\Omega}_h}\abs{\nabla u_h - \Tilde{R}\pi_h}^q \ dX + \int_{\hat{\Omega}_h}\abs{R\pi_h -\Tilde{R}\pi_h}^q \ dX\bigg)\\
        &= C\bigg(\int_{\hat{\Omega}_h}\abs{\nabla u_h - \Tilde{R}\pi_h}^q \ dX + \int_{\hat{\Omega}_h}\text{dist}^q(\Tilde{R}\pi_h, SO(3)) \ dX\bigg)\\
        &\leq C\bigg(\int_{\hat{\Omega}_h}\abs{\nabla u_h - \Tilde{R}\pi_h}^q \ dX + \int_{\hat{\Omega}_h}\text{dist}^q(\nabla u_h, SO(3)) \ dX\bigg)\\
        &\leq C K(u_h, \hat{\Omega}_h),
    \end{align*}
    namely we deduce (i).

\textbf{Step 6} (Showing (iii)):

First an auxiliary matrix $\Tilde{Q}$ is defined as the average of $R$ over $\hat{\omega}_h$:
\begin{equation*}
    \Tilde{Q} := \fint_{\hat{\omega}_h}R(X) \ dX.
\end{equation*}
By the Poincaré inequality and (ii) it follows that
\begin{equation}
    \int_{\hat{\omega}_h}\abs{R(X) - \Tilde{Q}}^q \ dX \leq C \int_{\hat{\omega}_h}\abs{\hat{\nabla} R(X)}^q \ dX \leq Ch^{-q-1}K(u_h,\hat{\Omega}_h).
    \label{4.37}
\end{equation}
Arguing as in Step 5, we set $Q := \Pi_{SO(3)}\Tilde{Q}$, which is well defined due to the previous estimate. Since for any $X \in \hat{\omega}_h$ it holds that
\begin{equation*}
     \abs{\Tilde{Q} - Q} = \text{dist}(\Tilde{Q}, SO(3)) \leq \abs{\Tilde{Q} - R(X)},
\end{equation*}
by \eqref{4.37} we obtain
\begin{align*}
    \int_{\hat{\omega}_h}\abs{R(X) - Q}^q \ dX &\leq \int_{\hat{\omega}_h}\abs{R(X) - \Tilde{Q}}^q + \abs{\Tilde{Q} - Q}^q \ dX\\
    &\leq 2 \int_{\hat{\omega}_h}\abs{R(X) - \Tilde{Q}}^q \ dX\\
    &\leq Ch^{-q-1}K(u_h,\hat{\Omega}_h),
\end{align*}
which concludes the proof.
\end{proof}

By the change of variables from \eqref{3.17}, we infer the following Corollary of Theorem \ref{Thm4.1}:

\begin{corollary}
\label{Cor4.1}
    Let $y_h := u_h \circ \Theta_h \circ z_h \in W^{1,p}(\Omega;\mathbb{R}^3)$ and suppose that $h^{-p}K(u_h,\hat{\Omega}_h)$ is sufficiently small. Set
    \begin{equation*}
        K_h(y_h,\Omega) := \int_\Omega \text{dist}^2(\nabla_h y_h M_h, SO(3)) \vee \text{dist}^p(\nabla_h y_h M_h, SO(3)) \ dx.
    \end{equation*}
    Then, there exist matrix fields $R_h \in W^{1,p}(\omega;\mathbb{R}^{3\times 3}) $ such that $R_h(x') \in SO(3)$ for all $x' \in \omega$, and matrices $Q_h \in SO(3)$ with the following properties:
    \begin{align}
        &(i) \ \norm{\nabla_h y_h M_h - R_h}_{L^q(\Omega;\mathbb{R}^{3\times 3})}^q \leq \textcolor{Green}{CK_h(y_h,\Omega)},
        \label{4.44}\\
        &(ii) \ \norm{\hat{\nabla} R_h}_{L^q(\omega;\mathbb{R}^{3\times 3})}^q \leq C\textcolor{Green}{h^{-q}}K_h(y_h,\Omega),
        \label{4.45}\\
        &(iii) \ \norm{R_h - Q_h}_{L^q(\omega;\mathbb{R}^{3\times 3})}^q \leq C\textcolor{Green}{h^{-q}}K_h(y_h,\Omega),
        \label{4.46}
    \end{align}
    where the constant $C$ is independent of $y_h$ and $h$.
\end{corollary}
\begin{proof}
    Set $R_h := R \circ \pi_h \circ \Theta_h\circ z_h$ and let $Q_h$ be constructed as $Q$ in the proof of Theorem \ref{Thm4.1}. By the Change-of-Variables formula, Theorem \ref{Thm4.1} and \eqref{4.38}, there holds
    \begin{align*}
        &\int_{\Theta_h\circ z_h(\Omega)}\abs{\nabla u_h - R\circ\pi_h}^q \ dX\\
        &= \int_\Omega\abs{((\nabla u_h) \circ\Theta_h\circ z_h) - R ( \pi_h(\Theta_h\circ z_h))}^q\abs{\det (\nabla \Theta_h\circ z_h)} \ dx\\
        &= h\int_\Omega\abs{\nabla_h y_h M_h - R_h}^q\kappa_h \ dx\\
        &\leq C \int_{\Theta_h\circ z_h(\Omega)}\text{dist}^2(\nabla u_h, SO(3)) \vee \text{dist}^p(\nabla u_h, SO(3)) \ dX\\
        &= Ch \int_\Omega \text{dist}^2(\nabla_h y_h M_h, SO(3)) \vee \text{dist}^p(\nabla_h y_h M_h, SO(3)) \ \kappa_h \ dx\\
        &\leq ChK_h(y_h,\Omega).
    \end{align*}
    
    Together with property \eqref{4.39}, this shows (i).\\~\ 

    \textcolor{Green}{For $(ii)$, we first observe that the area-element of $\hat{\omega}_h$ is $\norm{\partial_1(\Theta_h \circ z_h (x')) \times \partial_2(\Theta_h \circ z_h (x'))}_2 = \sqrt{1 + h^2 \norm{\nabla ' \theta(x')}^2}$. With this we can conclude
    \begin{align*}
        \int_\omega \abs{(\nabla R_h) \circ \Theta_h \circ z_h}^q \ dx' &\leq \int_\omega \abs{(\nabla R_h) \circ \Theta_h \circ z_h}^q\sqrt{1 + h^2 \norm{\nabla ' \theta(x')}^2} \ dx'\\
        &= \int_{\hat{\omega}_h}\abs{\hat{\nabla}R_h}^q \ dX \leq Ch^{-q-1}K_h(u_h,\hat{\Omega}_h)\\
        &\leq C h^{-q-1}h K_h(y_h,\Omega).
    \end{align*}
    An analogous estimate shows $(iii)$.
    }
\end{proof}

Utilizing \eqref{4.19}, we obtain:
\begin{corollary}
\label{Cor4.2}
    Let $y_h \in W^{1,p}(\Omega;\mathbb{R}^3)$. Suppose that $h^{-p}K(u_h,\hat{\Omega}_h)$ is sufficiently small and additionally that $E_h(y_h,m_h) \leq C$ for every $h > 0$.\\
    Then, there holds:
    \begin{align*}
        &(i) \ \norm{\nabla_h y_h M_h - R_h}_{L^q(\Omega;\mathbb{R}^{3\times 3})} \leq Ch^{\frac{\beta}{q}},\\
        &(ii) \ \norm{\nabla R_h}_{L^q(\omega;\mathbb{R}^{3\times 3})} \leq Ch^{\frac{\beta}{q} - 1},\\
        &(iii) \ \norm{R_h - Q_h}_{L^q(\omega;\mathbb{R}^{3\times 3})} \leq Ch^{\frac{\beta}{q} - 1},
    \end{align*}
    where the constant $C$ is independent of $y_h$ and $h$.
\end{corollary}

We now use this approximation by rotations result to show compactness of deformations and magnetizations. The proofs of the following theorems are essentially adaptations of the analogous results from \cite[Section 4]{BrescianiIncompressible} in the setting of plates.

\subsection*{Compactness of deformations}

\begin{theorem}
(Compactness of deformations)
\label{Thm 5.1}\\
    Consider a sequence of admissible states $(y_h,m_h)_h \subseteq \mathcal{Q}$ satisfying $E_h(y_h, m_h) \leq C$ for every $h > 0$. Then, there exists $u \in W^{1,2}(\omega;\mathbb{R}^2)$, $v \in W^{2,2}(\omega)$, a sequence of rotations $(Q_h)_h \subseteq SO(3)$ and a sequence of translation vectors $(c_h)_h \subseteq \mathbb{R}^3$, such that by setting $\Tilde{y}_h = T_h \circ y_h$, where $T_h: \mathbb{R}^3 \to \mathbb{R}^3$ is the rigid motion given by $T_h(\xi) := Q_h^T\xi - c_h$ for every $\xi \in \mathbb{R}^3$, up to subsequences there holds:
    \begin{align*}
        &\Tilde{y}_h \to \Theta_0 \circ z_0 := \lim_{h \to 0}\Theta_h\circ z_h \ \text{in} \ W^{1,p}(\Omega;\mathbb{R}^3),\\
        &U_h^{\Tilde{y}_h} \rightharpoonup u \ \text{weakly in} \ W^{1,2}(\omega;\mathbb{R}^2) \ \text{for some} \ u \in W^{1,2}(\omega;\mathbb{R}^2),\\
        &V_h^{\Tilde{y}_h} \to v \ \text{in} \ W^{1,2}(\omega) \ \text{for some} \ v \in W^{2,2}(\omega).
    \end{align*}
\end{theorem}

\begin{proof}
    Choosing $R_h$ and $Q_h$ as in Corollary \ref{Cor4.2}, we obtain the following estimates:
    \begin{align*}
        &\norm{\nabla_h y_h M_h - R_h}_{L^q(\Omega;\mathbb{R}^{3\times 3})} \leq Ch^{\frac{\beta}{q}},\\
        &\norm{R_h - Q_h}_{L^q(\omega;\mathbb{R}^{3\times 3})} \leq Ch^{\frac{\beta}{q} - 1}.
    \end{align*}
    Together this yields
    \begin{equation}
        \norm{\nabla_h y_h M_h - Q_h}_{L^q(\Omega;\mathbb{R}^{3\times 3})} \leq Ch^{\frac{\beta}{q} - 1}.
        \label{5.4}
    \end{equation}
    We set
    \begin{equation*}
        c_h := \fint_\Omega(Q_h^Ty_h - \Theta_h\circ z_h) \ dx.
    \end{equation*}

    Then by consecutively applying the Poincaré-inequality, the chain rule, the fact that $\nabla \Theta_h\circ z_h$ is bounded by \eqref{4.7} and \eqref{5.4}, we obtain
    \begin{align*}
        \norm{\Tilde{y}_h - \Theta_h\circ z_h}_{L^q(\Omega;\mathbb{R}^{3\times 3})} &\leq \norm{\nabla\Tilde{y}_h - \nabla(\Theta_h\circ z_h)}_{L^q(\Omega;\mathbb{R}^{3\times 3})}\\
        &= \norm{Q_h^T \nabla(u \circ \Theta_h\circ z_h) - \nabla(\Theta_h\circ z_h)}_{L^q(\Omega;\mathbb{R}^{3\times 3})}\\
        &= \norm{Q_h^T ((\nabla u) \circ \Theta_h\circ z_h)\nabla (\Theta_h \circ z_h) - \nabla(\Theta_h\circ z_h)}_{L^q(\Omega;\mathbb{R}^{3\times 3})}\\
        &\leq C \norm{Q_h^T\nabla_h y_h M_h - Id}_{L^q(\Omega;\mathbb{R}^{3\times 3})}\\
        &\leq Ch^{\frac{\beta}{q}},
    \end{align*}    

   \textcolor{Green}{where we also used the fact that,
    \begin{align*}
        \norm{\nabla (\Theta_h \circ z_h)}_{L^\infty(\Omega;\mathbb{R}^{3\times 3})} \leq C.
    \end{align*}}

    Since $\beta > 2p$, $q \in \{2,p\}$ and
    \begin{equation*}
        \frac{\beta}{q} \geq \frac{\beta}{p} > 2,
    \end{equation*}
    it follows that $\Tilde{y}_h \to \Theta_0 \circ z_0 \ \text{in} \ W^{1,p}(\Omega;\mathbb{R}^3)$, as $h \to 0^+$. 

    We obtain $U_h^{\Tilde{y}_h} \rightharpoonup u \ \text{weakly in} \ W^{1,2}(\omega;\mathbb{R}^2)$ and $V_h^{\Tilde{y}_h} \to v \ \text{in} \ W^{1,2}(\omega)$ \textcolor{Green}{ from equations $(32)$ and $(33)$ in \cite[Lemma 2]{Velcic}}.
\end{proof}

\subsection*{Compactness of magnetizations}

\begin{theorem}
\label{Thm 5.2}
    (Compactness of magnetizations)\\
    For a sequence of admissible maps $((y_h, m_h))_h \subseteq \mathcal{Q}$ with $E_h^{el}(y_h,m_h) + E_h^{exc}(y_h,m_h) \leq C$, for every $h > 0$ we set $\Tilde{y}_h := T_h \circ y_h$ and $\Tilde{m}_h := Q_h^Tm_h\circ T_h^{-1}$, with $Q_h \in SO(3)$ and $T_h: \mathbb{R}^3 \to \mathbb{R}^3 $ defined as in Theorem \ref{Thm 5.1}. Then, there exists a map $\lambda \in W^{1,2}(\omega;\mathbb{S}^2)$, such that, up to subsequences, there holds:
    \begin{align}
        &\Tilde{m}_h \circ \Tilde{y}_h \to \lambda \ \ \text{in} \ \ L^r(\Omega;\mathbb{R}^3) \ \ \text{for every} \ 1\leq r < \infty,
        \label{5.7}\\
        &\Tilde{\mu}_h := (\chi_{\Omega^{\Tilde{y}_h}}m_h)\circ\Theta_h\circ z_h \to \chi_{\Omega}\lambda \ \ \text{in} \ \ L^r(\mathbb{R}^3;\mathbb{R}^3) \ \ \text{for every} \ 1\leq r < \infty.
        \label{5.8}
    \end{align}
\end{theorem}

\begin{proof}
    The proof is an adaptation of \cite[Proposition 4.2]{BrescianiIncompressible}, for shallow shells. It follows the same four steps approach. Steps 1 and 2 require arguments adapted for shallow shells and are presented, whereas Steps 3 and 4 require no adaptation and are therefore obmitted.

    \textbf{Step 1} (Existence of an inner cylinder):

   From Theorem \ref{Thm 5.1} we know that

   \begin{equation*}
       \norm{\Tilde{y}_h - \Theta_h\circ z_h}_{L^p(\Omega;\mathbb{R}^{3\times 3})} \leq Ch^{\frac{\beta}{p}} := \gamma_h,
   \end{equation*}
   and by Sobolev embeddings,
   \begin{equation*}
       \norm{\Tilde{y}_h - \Theta_h\circ z_h}_{C^0(\overline{\Omega};\mathbb{R}^3)} \leq \gamma_h.
   \end{equation*}

   Now, define for $a, b > 0$
   \begin{align*}
       &\omega^a := \{ x' \in \omega \ : \ \text{dist}(x',\partial \omega) > a \},\\
       &\Omega^a_b := \omega^a \times \big(-\frac{b}{2},\frac{b}{2}\big),\\
       &\omega^{-a} := \{ x' \in \mathbb{R}^2 \ : \ \text{dist}(x',\partial \omega) < a \},\\
       &\Omega^{-a}_b := \omega^{-a} \times \big(-\frac{b}{2},\frac{b}{2}\big).
   \end{align*}
   
   Then, we set $\hat{\Omega}_{bh}^a = \Theta_h(z_h(\Omega^a_{b}))$ and $\hat{\Omega}_{bh}^{-a} = \Theta_h(z_h(\Omega_b^{-a}))$.

   In the next argument we write $\text{dist}_{\hat{\omega}_h}$ for the geodesic distance on $\hat{\omega}_h$ and use that since $\Theta_h \circ z_h$ is a $C^1$-Diffeomorphism, it is a bi-Lipschitz map.

   We fix a $\varepsilon > 0$ and $\tau \in (0,1)$. Then, for every $x \in \Omega_{\tau}^\varepsilon$, using the fact that $\Theta_h$ is bi-Lipschitz \textcolor{Green}{on $\overline{\Omega}_h$} and the definition of $\Omega_{\tau}^\varepsilon$, there holds
\textcolor{Green}{
   \begin{align*}
       \text{dist}(\Theta_h(z_h(x)), \partial \hat{\Omega}_h) &\geq C\text{dist}(x,\partial \Omega_h)\\
       &\geq C (\text{dist}(x',\partial \omega) \wedge h(\frac{1}{2} - \abs{x_3}))\\
       &> C(\varepsilon \wedge h\frac{1 - \tau}{2}).
   \end{align*}
   }

   Recalling $\gamma_h = Ch^{\frac{\beta}{p}}$ and $\frac{\beta}{p} > 2$, we observe that for sufficiently small $h$ there holds\\ $C(\varepsilon \wedge h\frac{1 - \tau}{2}) > \gamma_h$, which implies
   \begin{equation*}
       \norm{\Tilde{y}_h - \Theta_h}_{C^0(\overline{\Omega};\mathbb{R}^3)} < \text{dist}(\Theta_h(z_h(x)), \partial \hat{\Omega}_h).
   \end{equation*}

   From \ref{D5}, it follows that
   \begin{align*}
       &\Theta_h(z_h(x)) \notin \Tilde{y}_h(\partial \Omega),\\
       &\text{deg}(\Tilde{y}_h, \Omega, \Theta_h(z_h(x))) = \text{deg}(\Theta_h \circ z_h, \Omega, \Theta_h(z_h(x))) = 1.
   \end{align*}
   By \ref{D4}, this implies $\Theta_h(z_h(x)) \in \Tilde{y}_h(\Omega)$.

   Since $\varepsilon$ and $\tau$ were arbitrarily chosen, we can formulate the just shown argument in the following way:
   \begin{equation}
       \forall \ \varepsilon > 0, \forall \ \tau \in (0,1), \exists \ \overline{h}(\varepsilon,\tau) > 0: \forall \ h \in (0,\overline{h}(\varepsilon,\tau)], \  \hat{\Omega}^\varepsilon_{\tau h} \subseteq \Omega^{\Tilde{y}_h}.
       \label{5.15}
   \end{equation}

   \textbf{Step 2} (Convergence of magnetizations):

   By the chain rule, there holds
   \begin{align*}
       \nabla \Tilde{m}_h = Q_h^T\nabla m_h \circ T_h^{-1}Q_h,
   \end{align*}
   and applying the Change-of-Variable formula, we obtain

   \begin{equation*}
        \begin{aligned}
            E_h^{exc}(\Tilde{y}_h, \Tilde{m}_h) &= \frac{\alpha}{h} \int_{\Omega^{\Tilde{y}_h}}\abs{\nabla \Tilde{m}_h}^2 \ d\xi\\
            &= \frac{\alpha}{h}\int_{T_h(\Omega^{y_h})}\abs{\nabla m_h \circ T_h^{-1}}^2 \ d\xi\\
            &= \frac{\alpha}{h}\int_{\Omega^{y_h}}\abs{\nabla m_h}^2 \ d\xi\\
            &= E_h^{exc}(y_h, m_h).
        \end{aligned}
   \end{equation*}

   Now we fix $\varepsilon > 0$ and $\tau \in (0,1)$ and apply \eqref{5.15}, so that for $h < \overline{h}(\varepsilon,\tau)$ we infer $\hat{\Omega}^\varepsilon_{\tau h} \subseteq \Omega^{\Tilde{y}_h}$. Using this, we obtain

   \begin{align*}
       C \geq E_h^{exc}(\Tilde{y}_h, \Tilde{m}_h) &= \frac{\alpha}{h}\int_{\Omega^{\Tilde{y}_h}}\abs{\nabla \Tilde{m}_h}^2 \ d\xi\\
       &\geq \frac{\alpha}{h}\int_{\hat{\Omega}^\varepsilon_{\tau h}}\abs{\nabla \Tilde{m}_h}^2 \ d\xi.
   \end{align*}
   
   Also, $\Tilde{m}_{h_{|\hat{\Omega}^\varepsilon_{\tau h}}}$ or analogously
   \begin{align*}
       \hat{m}_h := \Tilde{m}_h \circ\Theta_h \circ z_h(\Omega^\varepsilon_\tau),
   \end{align*}
   is well-defined and a map in $W^{1,2}(\Omega^\varepsilon_\tau; \mathbb{S}^2) $. 
   By a straightforward calculation for $x \in \Omega^\varepsilon_\tau$ and $\xi = (\Theta_h\circ z_h) (x)$ with the notation from \eqref{3.16}, we obtain
   \begin{align*}
       (\nabla\Tilde{m}_h)(\xi) &= (\nabla \Tilde{m}_h) \circ \Theta_h \circ z_h (x)\\
       &= ((\nabla\Theta_h)^{-T}\circ z_h)(x)(\nabla(\Tilde{m}_h\circ\Theta_h)\circ z_h)(x)\\
       &= M_h^T(x)(\nabla_h(\Tilde{m}_h \circ \Theta_h \circ z_h))(x)\\
       &= M_h^T(x)(\nabla_h\hat{m}_h)(x)\\
       &= (M_h^T \circ(\Theta_h\circ z_h)^{-1})(\xi)(\nabla_h\hat{m}_h\circ(\Theta_h\circ z_h)^{-1})(\xi).
   \end{align*}
   
   Using this identity together with $\abs{(M_h^T \circ(\Theta_h\circ z_h)^{-1})(\xi)} = \abs{M_h^T(x)}$, property \eqref{4.41} and the Change-of-Variable formula, we infer
   \begin{align*}
       &\frac{\alpha}{h}\int_{\hat{\Omega}^\varepsilon_{\tau h}}\abs{\nabla \Tilde{m}_h}^2 \ d\xi\\
       &= \frac{\alpha}{h}\int_{\hat{\Omega}^\varepsilon_{\tau h}}\abs{((\nabla\Theta_h)^{-T}\circ z_h \circ(\Theta_h\circ z_h)^{-1})\nabla_h\hat{m}_h\circ(\Theta_h\circ z_h)^{-1}}^2 \ d\xi\\
       &\geq C\frac{\alpha}{h}\int_{\hat{\Omega}^\varepsilon_{\tau h}}\abs{(\nabla_h\hat{m}_h\circ(\Theta_h\circ z_h)^{-1})}^2 \ d\xi\\
       &= C\frac{\alpha}{h}\int_{\Omega^\varepsilon_\tau}\abs{\nabla_h \hat{m}_h}^2\abs{\det \nabla(\Theta_h \circ z_h)} \ dx\\
       &\geq C\alpha\int_{\Omega^\varepsilon_\tau}\abs{\nabla_h \hat{m}_h}^2 \ dx.
   \end{align*}

   To sum up, we have shown that
   \begin{equation}
       \alpha\int_{\Omega^\varepsilon_\tau}\abs{\nabla_h \hat{m}_h}^2 \ dx \leq C.
       \label{5.17}
   \end{equation}
    From this estimate it firstly follows with
    \begin{align*}
        \norm{\nabla \hat{m}_h}_{L^2(\Omega^\varepsilon_\tau;\mathbb{R}^3)} \leq \norm{\nabla_h \hat{m}_h}_{L^2(\Omega^\varepsilon_\tau;\mathbb{R}^3)} \leq C,
    \end{align*}
    that
    \begin{equation*}
        \hat{m}_h \rightharpoonup \lambda \ \ \text{in} \ \ W^{1,2}(\Omega^\varepsilon_\tau;\mathbb{R}^3),
    \end{equation*}
    for some $ \lambda \in W^{1,2}(\Omega^\varepsilon_\tau;\mathbb{R}^3)$ and secondly from 
    \begin{align*}
        \norm{h^{-1}\partial_3 \hat{m}_h}_{L^2(\Omega^\varepsilon_\tau;\mathbb{R}^3)} \leq C,
    \end{align*}
    that
    \begin{align*}
       &\partial_3 \hat{m}_h \rightharpoonup 0 \ \ \text{in} \ \ L^2(\Omega^\varepsilon_\tau;\mathbb{R}^3),\\
        &h^{-1}\partial_3 \hat{m}_h \rightharpoonup k \ \ \text{in} \ \ L^2(\Omega^\varepsilon_\tau;\mathbb{R}^3),
    \end{align*}
    for some $ k \in L^2(\Omega^\varepsilon_\tau;\mathbb{R}^3)$.

    Since the bound from \eqref{5.17} is independent of $\varepsilon$ and $\tau$, so is the convergence. Together with the fact that $\lambda$ does not depend on the third component $x_3$, it follows that $\lambda \in W^{1,2}_{loc}(\omega;\mathbb{S}^2)$ and $k \in L^2_{loc}(\Omega;\mathbb{R}^3)$. Additionally, by lower semi-continuity of the norm it follows that

    \begin{align*}
        C &\geq \liminf_{h \to 0^+}\int_{\Omega^\varepsilon_\tau}\abs{\nabla_h \hat{m}_h}^2 \ dx\\
        &\geq \int_{\Omega^\varepsilon_\tau}\abs{(\nabla'\lambda, k)}^2 \ dx\\
        &\geq \int_{\Omega^\varepsilon_\tau}\abs{\nabla'\lambda}^2 \ dx + \int_{\Omega^\varepsilon_\tau}\abs{k}^2 \ dx\\
        &\geq \tau \int_{\omega^\varepsilon}\abs{\nabla'\lambda}^2 \ dx.
    \end{align*}

    Therefore by letting $\varepsilon \to 0^+$ and $\tau \to 1^-$ it follows that $\nabla'\lambda \in L^2(\omega;\mathbb{R}^2)$, so that $\lambda \in W^{1,2}(\omega;\mathbb{S}^2)$.

    \textcolor{Green}{From this, by steps three and four in the proof of \cite[Proposition 4.2]{BrescianiIncompressible} the convergences in \eqref{5.7} and \eqref{5.8} follow.}

\end{proof}

%%%%%%%%%%%%%%%%%%%%%%%%%%%%%%%%%%%%%%%%%%%%%%%%%%%%%%%%%%%%%%%%%%%%%%%%

\section{Lower bound}
\label{s:7}

This section is dedicated to the second part of the proof of Theorem \ref{thm1}. We split it into three parts corresponding to the three parts of the magnetoelastic energy. We adapt arguments of \cite[Section 4]{BrescianiIncompressible}.

\subsection*{Lower bound for the exchange energy}

\begin{theorem}
    (Lower bound for the exchange energy)
    \label{Thm 6.1}\\
    Let $((y_h,m_h))_h \subseteq \mathcal{Q}$, be such that $E^{el}_h(y_h,m_h) + E^{exc}_h(y_h,m_h) \leq C $, for all $h >0$. Let $\lambda \in W^{1,2}(\omega; \mathbb{S}^2) $ be as in Theorem \ref{Thm 5.2}. Then,

    \begin{equation*}
        E^{exc}(\lambda) \leq \liminf_{h \to 0^+}E^{exc}_h(y_h,m_h).
    \end{equation*}
\end{theorem}
\begin{proof}
    Recall that
    \begin{equation*}
        E^{exc}_h(y_h,m_h) = \frac{\alpha}{h}\int_{\Omega^{y_h}}\abs{\nabla m_h}^2 \ d\xi.
    \end{equation*}

    As in the proof of Theorem \ref{Thm 5.2}, set $\Tilde{y_h} = T_h\circ u_h$, $\Tilde{m}_h = Q_h^Tm_h\circ T_h^{-1}$ and $\hat{m}_h = \Tilde{m}_h \circ\Theta_h\circ z_{h|_\Omega}$. Now recalling  step 2 in said proof, we fix $\varepsilon > 0$ and $\tau \in (0,1)$ such that
    \begin{align*}
        &\hat{m}_h \rightharpoonup \lambda \ \ \text{in} \ W^{1,2}(\Omega_\tau^\varepsilon ; \mathbb{R}^3) \ \ \text{for some} \ \lambda \in W^{1,2}(\omega;\mathbb{S}^2)\\ 
        &h^{-1}\partial_3 \hat{m}_h \rightharpoonup k \ \ \text{in} \ L^2(\Omega_\tau^\varepsilon ; \mathbb{R}^3) \ \ \text{for some} \ k \in L^2_{loc}(\Omega;\mathbb{R}^3).
    \end{align*}

    Analogous to the calculation leading to \eqref{5.17}, we infer

    \begin{align*}
        \liminf_{h \to 0^+} E_h^{exc}(y_h,m_h) &= \liminf_{h \to 0^+} E_h^{exc}(\Tilde{y}_h,\Tilde{m}_h)\\
        &= \liminf_{h \to 0^+}\frac{\alpha}{h}\int_{\Omega^{\Tilde{y}_h}}\abs{\nabla \Tilde{m}_h}^2 \ d\xi\\
        &\geq \liminf_{h \to 0^+}\frac{\alpha}{h}\int_{\hat{\Omega}^\varepsilon_{h\tau}}\abs{\nabla \Tilde{m}_h}^2 \ d\xi\\
        &\geq \liminf_{h \to 0^+}\alpha\int_{\Omega^\varepsilon_{\tau}}\abs{\nabla_h \hat{m}_h}^2 \ dx\\
        &\geq \alpha \int_{\Omega^\varepsilon_\tau}\abs{\nabla ' \lambda}^2 \ dx + \alpha \int_{\Omega^\varepsilon_\tau}\abs{k}^2 \ dx\\
        &\geq \alpha \int_{\Omega^\varepsilon_\tau}\abs{\nabla ' \lambda}^2 \ dx\\
        &= \alpha \tau \int_{\omega^\varepsilon}\abs{\nabla ' \lambda}^2 \ dx'.
    \end{align*}

    Taking $\varepsilon \to 0^+$ and $\tau \to 1^-$ concludes the proof.
    
\end{proof}

\subsection*{Lower bound for the elastic energy}

As in \cite{BrescianiIncompressible}, we approximate the incompressible energy from below by adding a penalization term that forces the determinant to be close to 1. Namely, for every $k \in \mathbb{N}$, we define the approximated energy density as

\begin{equation*}
    W^k(F,\nu) := W(F,\nu) + \frac{k}{2}(\det F - 1)^2,
\end{equation*}
for all $F \in \mathbb{R}^{3\times 3} $, $\nu \in \mathbb{S}^2$. Note that by definition $W^{inc} \geq W^k \geq W$, and $W^k$ satisfies \ref{W1} and \ref{W2}. We also infer the following Taylor expansion:

\begin{equation}
    W^k(Id + G, \nu) = \frac{1}{2}Q_3^k(G,\nu) + \omega^k(G,\nu),
    \label{6.3}
\end{equation}
for all $G \in \mathbb{R}^{3\times 3} $ with $\abs{G} < \delta$, where $\delta$ is taken from \ref{W4}, and for all $\nu \in \mathbb{S}^2$, with
\begin{align*}
    &Q_3^k(G,\nu) = Q_3(G,\nu) + k(\text{tr}(G))^2,\\
    &\omega^k(G,\nu) = \omega(G,\nu) + k \gamma(\abs{G}^2),
\end{align*}
for $\gamma(t) = o(t^2)$ as $t \to 0^+$.

We further define

\begin{equation*}
    Q_2^k(H, \nu) := \min \bigg\{Q_3^k\bigg(\left(\begin{array}{c|c}
    	H & 0'\\ 
    	\hline 
    	(0')^T & 0
    \end{array}\right) 
    + c \otimes e_3 + e_3 \otimes c, \nu \bigg) : c \in \mathbb{R}^3\bigg\},
\end{equation*}

for all $H \in \mathbb{R}^{2\times 2} $ and $\nu \in \mathbb{S}^2$. 

By \cite[Lemma 2.1]{ContiEstimate}, there holds

\begin{equation}
    Q_2^k(H, \nu) \geq Q_2^{inc}(H, \nu) - \frac{C}{\sqrt{k}}\abs{H}^2.
    \label{6.7}
\end{equation}

We are now in a position to show the lower bound for the elastic energy.

\begin{theorem}
    (Lower bound for the elastic energy)
    \label{Thm 6.2}\\
    Let $((y_h,m_h))_h \subseteq \mathcal{Q}$ be such that $E^{el}_h(y_h,m_h) + E^{exc}_h \leq C$ for all $h > 0$. Then, for $u \in W^{1,2}(\omega;\mathbb{R}^2) $, $v \in W^{2,2}(\omega)$ and $\lambda \in W^{1,2}(\omega;\mathbb{S}^2)$ as in Theorem \ref{Thm 5.1} and Theorem \ref{Thm 5.2}, there holds
    \begin{equation*}
        E^{el}(u,v,\lambda) \leq \liminf_{h\to 0^+}E^{el}_h(y_h,m_h).
    \end{equation*}
\end{theorem}

\begin{proof}
    We adapt the proof of \cite[Proposition 4.4]{BrescianiIncompressible}.

    First, as in Theorem \ref{Thm 5.1}, we set $\Tilde{y}_h = T_h \circ y_h$ and $\Tilde{m}_h = Q_h^Tm_h\circ T_h^{-1}$. Also, set $F_h := \nabla_h y_h M_h$, $\Tilde{F}_h := Q_h^TF_h$ and $\Tilde{R}_h = Q_h^TR_h$, where $R_h$ is given by Corollary \ref{Cor4.2}. We define 
    \begin{align*}
        &G_h := \frac{1}{h^{\frac{\beta}{2}}}(\Tilde{R}_h^T\Tilde{F}_h - Id),\\
        &X_h := \{x \in \Omega : \abs{G_h(x)} \leq h^{-\frac{\beta}{4}}\},\\
        &\chi_h := \chi_{X_h}.
    \end{align*}
    
    \textcolor{Green}{From \cite[(59) and (60)]{Velcic}, we obtain}

    \begin{equation*}
        G_h \rightharpoonup G \ \ \text{in} \ L^2(\Omega;\mathbb{R}^{3\times 3}),
    \end{equation*}

    with
    \begin{align}
        G''(x) &= K(x') + L(x')x_3,
        \label{6.10}
    \end{align}

    for almost every $x \in \Omega$, where $K := \text{sym}(\nabla'u + \nabla'v \otimes \nabla'\theta) \in \mathbb{R}^{2\times 2} $ and $L := -(\nabla ')^2v \in \mathbb{R}^{2\times 2} $. Here $u$ and $v$ are the averaged limiting displacements identified in \eqref{3.23} and \eqref{3.24}.

    Since $E^{el}_h(y_h,m_h) \leq C$, we have that $W^{inc} = W^k = W$.

    \textcolor{Green}{Setting $\nu_h := m_h \circ y_h$ and $\Tilde{\nu}_h := \Tilde{m}_h \circ \Tilde{y}_h$, by \ref{W1}} and the Taylor expansion \eqref{6.3} we obtain

    \begin{equation}
        \begin{aligned}
            \chi_h W^{inc}(F_h,\nu_h) &= \chi_h W^{k}(F_h,\nu_h) = \chi_h W^{k}(Q_h^TF_h,Q_h^T\nu_h) = \chi_h W^{k}(\Tilde{F}_h,\Tilde{\nu}_h)\\
            &= \chi_h W^{k}(\Tilde{R}_h^T\Tilde{F}_h,\Tilde{R}_h^T\Tilde{\nu}_h) = \chi_h W^k(Id + h^{\frac{\beta}{2}}G_h, \Tilde{R}_h^T\Tilde{\nu}_h)\\
            &= \frac{1}{2}\chi_hQ_3^k(\chi_h^{\frac{\beta}{2}}G_h, \Tilde{R}_h^T\Tilde{\nu}_h) + \chi_h\omega^k(h^{\frac{\beta}{2}}G_h, \Tilde{R}_h^T\Tilde{\nu}_h)\\  
            &= \frac{h^\beta}{2}Q_3^k(\chi_hG_h,\Tilde{R}_h^T\Tilde{\nu}_h) + \omega^k(h^{\frac{\beta}{2}}\chi_h G_h, \Tilde{R}_h^T\Tilde{\nu}_h).
        \end{aligned}
        \label{6.11}
    \end{equation}

    In the last step we used the fact that
    \begin{align*}
        Q_3^k(h^{\frac{\beta}{2}}G_h,\nu) &= Q_3(h^{\frac{\beta}{2}}G_h,\nu) + k(\text{tr}(h^{\frac{\beta}{2}}G_h))^2\\
        &= \mathbb{C}^\nu h^{\frac{\beta}{2}}G_h : h^{\frac{\beta}{2}}G_h + k(\text{tr}(h^{\frac{\beta}{2}}G_h))^2\\
        &= h^\beta Q_3^k(G_h,\nu).
    \end{align*}

    By \eqref{6.11} and \eqref{4.39}, we deduce
    \begin{equation*}
        \begin{aligned}
            E^{el}(y_h,m_h) &= \frac{1}{h^\beta}\int_\Omega W^k(F_h,\nu_h) \ \kappa_h \ dx\\
            &\geq C\frac{1}{h^\beta}\int_\Omega \chi_h W^k(F_h,\nu_h) \ dx\\
            &= C\frac{1}{h^\beta}\int_\Omega \frac{h^\beta}{2}Q_3^k(\chi_hG_h,\Tilde{R}_h^T\Tilde{\nu}_h) + \omega^k(h^{\frac{\beta}{2}}\chi_h G_h, \Tilde{R}_h^T\Tilde{\nu}_h) \ dx.
        \end{aligned}
    \end{equation*}

    \textcolor{Green}{This  result is analogous to equation \cite[(4.43)]{BrescianiIncompressible}. By \cite[(4.46) and (4.47)]{BrescianiIncompressible}} we obtain
    \begin{equation}
        \liminf_{h\to 0^+}\int_\Omega Q_3^k(\chi_h G_h,\Tilde{R}_h^T\Tilde{\nu}_h) \ dx \geq \int_\Omega Q_3^k(G,\lambda) \ dx
        \label{6.13}
    \end{equation}

    and
    \begin{equation}
        \lim_{h \to 0^+}\frac{1}{h^\beta}\int_\Omega \omega^k(h^{\frac{\beta}{2}}\chi_h G_h, \Tilde{R}_h^T\Tilde{\nu}_h) \ dx = 0.
        \label{6.14}
    \end{equation}

    Using \eqref{4.42}, by \eqref{6.13} and \eqref{6.14} we have that
    \begin{equation*}
        \liminf_{h \to 0^+}E^{el}_h \geq \frac{1}{2}\int_\Omega Q_3^k(G,\lambda) \ dx \geq \frac{1}{2}\int_\Omega Q_2^k(G'',\lambda) \ dx,
    \end{equation*}

    and by \eqref{6.7},
    \begin{equation}
        \liminf_{h \to 0^+}E^{el}_h \geq \frac{1}{2}\int_\Omega Q_2^{inc}(G'',\lambda) \ dx - \frac{C}{2\sqrt{k}}\int_\Omega \abs{G''}^2 \ dx.
        \label{6.16}
    \end{equation}

    With \eqref{6.10}, together with the fact that $Q_2^{inc}$ is additive and
    \begin{align*}
        Q_2^{inc}(L(x')x_3, \lambda) = x_3^2Q_2^{inc}(L(x'), \lambda),
    \end{align*}
    
    we obtain
    \begin{equation}
        \begin{aligned}
           \int_\Omega Q_2^{inc}(G'',\lambda) \ dx &= \int_\Omega Q_2^{inc}(K(x'),\lambda) + x_3^2Q_2^{inc}(L(x'),\lambda) \ dx\\
            &= \int_{-\frac{1}{2}}^{\frac{1}{2}}\int_\omega Q_2^{inc}(K(x'),\lambda) + x_3^2Q_2^{inc}(L(x'),\lambda) \ dx' \ dx_3\\
            &= \int_\omega Q_2^{inc}(K(x'),\lambda) \ dx' + \frac{1}{12}\int_\omega Q_2^{inc}(L(x'),\lambda) \ dx'.
        \end{aligned}
        \label{6.17}
    \end{equation}

    With \eqref{6.17} and taking $k \to \infty$ in \eqref{6.16}, we conclude \textcolor{Green}{by a diagonal argument} that
    \begin{equation*}
        \begin{aligned}
            \liminf_{h \to 0^+}E^{el}_h &\geq \frac{1}{2} \int_\omega Q_2^{inc}(K, \lambda) \ dx' + \frac{1}{24}\int_\omega Q_2^{inc}(L, \lambda) \ dx'\\
            &= E^{el}(u, v, \lambda).
        \end{aligned}
    \end{equation*}

\end{proof}

\subsection*{Convergence of the magnetostatic energy}

We show the convergence of the magnetostatic energy. This result will also be instrumented to prove the convergence of the recovery sequence in the next section.

The weak formulation corresponding to \eqref{3.2} is the following: 

Find $\psi_m \in V^{1,2}(\mathbb{R}^3) $ such that
\begin{equation}
    \forall \ \varphi \in V^{1,2}(\mathbb{R}^3), \ \int_{\mathbb{R}^3}\nabla \psi_m \cdot \nabla \varphi \ d\xi = \int_{\mathbb{R}^3}\chi_{\Omega^y}m\cdot \nabla \varphi \ d\xi,
    \label{6.19}
\end{equation}

where
\begin{equation*}
    V^{1,2}(\mathbb{R}^3) := \big\{ \varphi \in L^2_{loc}(\mathbb{R}^3) : \nabla \varphi \in L^2(\mathbb{R}^3;\mathbb{R}^3)\big\}.
\end{equation*}

We cite \cite[Lemma 4.6]{BrescianiIncompressible} giving existence of such a stray field potential $\psi_m$:

\begin{lemma}
    \cite[Lemma 4.2]{BrescianiIncompressible}(Weak solutions of the Maxwell equation)
    \label{L 6.1}\\
    Let $(y,m) \subseteq \mathcal{Q}$. The Maxwell equation \eqref{3.2} admits a weak solution $\psi_m \in V^{1,2}(\mathbb{R}^3) $ which is unique up to additive constants and satisfies the following stability estimate:
    \begin{equation}
        \norm{\nabla \psi_m}_{L^2(\mathbb{R}^3;\mathbb{R}^3)} \leq \norm{\chi_{\Omega^y}m}_{L^2(\mathbb{R}^3;\mathbb{R}^3)}.
        \label{6.20}
    \end{equation}
    Moreover, such a weak solution admits the following variational characterization:
    \begin{equation*}
        \psi_m \in \text{argmin}\big\{ \int_{\mathbb{R}^3}\abs{\nabla \varphi - \chi_{\Omega^y}m}^2 \ d\xi : \varphi \in V^{1,2}(\mathbb{R}^3) \big\}. 
    \end{equation*}
\end{lemma}

The proof of the following theorem is adapted from \cite[Proposition 4.5]{BrescianiIncompressible}.

\begin{theorem}
    (Convergence of the magnetostatic energy)
    \label{Thm 6.3}\\
    Let $((y_h,m_h))_h \subseteq \mathcal{Q}$ be such that $E_h(y_h,m_h) \leq C$ for every $h > 0$. Let $\lambda \in W^{1,2}(\omega;\mathbb{S}^2)$ be the map identified in Theorem \ref{Thm 5.2}. Then, up to subsequences, there holds:
    \begin{equation*}
        \lim_{h\to 0^+}E_h^{mag}(y_h,m_h) = E^{mag}(\lambda).
    \end{equation*}
\end{theorem}

\begin{proof}
    \textbf{Step 1} (Invariance of the energy under rigid motions):
    
    We write again $\Tilde{y}_h = T_h\circ y_h$, $\Tilde{m}_h = Q_h^Tm_h \circ T_h^{-1}$. Also denote by $\psi_h$ and $\Tilde{\psi}_h$ the stray field potentials corresponding to $(y_h,m_h)$ and $(\Tilde{y}_h,\Tilde{m}_h)$ respectively. By Lemma \ref{L 6.1} they exist and are weak solutions to 
    \begin{equation}
        \begin{aligned}
            \Delta \psi_h &= \text{div} (\chi_{\Omega^{y_h}}m_h) \ \ \text{in} \ \mathbb{R}^3,\\
            \Delta \Tilde{\psi}_h &= \text{div} (\chi_{\Omega^{\Tilde{y}_h}}\Tilde{m}_h) \ \ \text{in} \ \mathbb{R}^3.
        \end{aligned}
        \label{6.23}
    \end{equation}

    For every $\varphi \in V^{1,2}(\mathbb{R}^3)$, we obtain by applying the chain rule, Change-of-Variable formula and the weak formulation of \eqref{6.19}
    \begin{equation}
        \begin{aligned}
            \int_{\mathbb{R}^3}\nabla(\psi_h \circ T_h^{-1}) \cdot \nabla \varphi \ d\xi &= \int_{\mathbb{R}^3}Q_h^T\nabla \psi_h \cdot \nabla \varphi \ d\xi\\
            &= \int_{\mathbb{R}^3} \nabla \psi_h \cdot \nabla(\varphi \circ T_h) \ d\xi\\
            &= \int_{\mathbb{R}^3} \chi_{\Omega^{y_h}}m_h \cdot Q_h \nabla\varphi \circ T_h \ d\xi\\
            &= \int_{\mathbb{R}^3} (\chi_{\Omega^{y_h}}Q_h^Tm_h) \circ T_h^{-1} \cdot \nabla\varphi \ d\xi\\
            &= \int_{\mathbb{R}^3} \chi_{\Omega^{\Tilde{y}_h}}\Tilde{m}_h \cdot \nabla \varphi \ d\xi.
        \end{aligned}
        \label{6.24}
    \end{equation}

    Equality \eqref{6.24} shows that $\psi_h \circ T_h^{-1}$ is a weak solution to the second equation in \eqref{6.23}. We can therefore assume that the solutions coincide, namely $\Tilde{\psi}_h = \psi_h \circ T_h^{-1}$. Using this, by the chain rule and Change-of-Variable formula, we obtain
    \begin{equation*}
        \begin{aligned}
            E^{mag}_h(\Tilde{y}_h,\Tilde{m}_h) &= \frac{1}{2h}\int_{\mathbb{R}^3} \abs{\nabla \Tilde{\psi}_h}^2 \ d\xi\\
            &= \frac{1}{2h}\int_{\mathbb{R}^3} \abs{\nabla \psi_h \circ T_h^{-1}}^2 \ d\xi\\
            &= \frac{1}{2h}\int_{\mathbb{R}^3} \abs{\nabla \psi_h}^2 \ d\xi\\
            &= E^{mag}_h(y_h,m_h).
        \end{aligned}
    \end{equation*}

    \textbf{Step 2} (Convergence of the stray field potential):

    Using the weak formulation of \eqref{6.23} with $\varphi = \Tilde{\psi}_h$ we infer
    \begin{equation*}
        \begin{aligned}
            E_h^{mag}(\Tilde{y}_h, \Tilde{m}_h) &= \frac{1}{2h}\int_{\mathbb{R}^3} \abs{\nabla \Tilde{\psi}_h}^2 \ d\xi\\
            &= \frac{1}{2h}\int_{\mathbb{R}^3} \chi_{\Omega^{\Tilde{y}_h}}\Tilde{m}_h.
        \end{aligned}
    \end{equation*}

    Setting $\hat{\psi}_h = \Tilde{\psi}_h \circ \Theta_h \circ z_h$ and $\Tilde{\mu}_h = (\chi_{\Omega^{\Tilde{y}_h}}\Tilde{m}_h) \circ \Theta_h \circ z_h$ and using $\Theta_h(z_h(\mathbb{R}^3)) = \mathbb{R}^3$, by the Change-of-Variable formula, there holds
    \begin{equation}
        \begin{aligned}
            \frac{1}{2h}\int_{\mathbb{R}^3} \chi_{\Omega^{\Tilde{y}_h}}\Tilde{m}_h &= \frac{1}{2}\int_{\mathbb{R}^3} \Tilde{\mu}_h \cdot (\nabla \Tilde{\psi}) \circ \Theta_h \circ z_h \ \kappa_h \ dx\\
            &= \frac{1}{2}\int_{\mathbb{R}^3}\Tilde{\mu}_h \cdot M_h^T\nabla_h \hat{\psi}_h \ \kappa_h \ dx.
        \end{aligned}
        \label{6.27}
    \end{equation}

    We now want to show that $\hat{\psi}_h \to 0$ in $V^{1,2}(\mathbb{R}^3)$. For this, notice that by \eqref{3.4} there holds $\det ((\nabla \Tilde{\psi}) \circ \Theta_h \circ z_h) = 1$. Using this, together with \eqref{4.38}, we obtain
    \begin{equation}
        \begin{aligned}
            \mathcal{L}^3(\Omega^{\Tilde{y}_h}) &= \int_{\Omega}\det \nabla \Tilde{y}_h \ dx\\
            &= h \int_\Omega \kappa_h \ dx\\
            &\leq h C
        \end{aligned}
        \label{6.28}
    \end{equation}

    By \eqref{6.28} and \eqref{6.20} together with the fact that $\Tilde{m}_h$ has values in $\mathbb{S}^2$, we conclude that
    \begin{equation}
        \begin{aligned}
            \int_{\mathbb{R}^3} \abs{\nabla \Tilde{\psi}_h}^2 \ d\xi &\leq \int_{\mathbb{R}^3}\abs{\chi_{\Omega^{\Tilde{y}_h}}\Tilde{m}_h}\\
            &= \mathcal{L}^3(\Omega^{\Tilde{y}_h})\\
            &\leq hC.
        \end{aligned}
        \label{6.29}
    \end{equation}

    By \eqref{4.39}, the Change-of-Variable formula and \eqref{6.29} we find
    \begin{align*}
        C\norm{M_h^T\nabla_h \hat{\psi}_h}_{L^2(\mathbb{R}^3;\mathbb{R}^3)} &\leq \int_{\mathbb{R}^3}\abs{M_h^T\nabla_h \hat{\psi}_h}^2\kappa_h \ dx\\
        &= \frac{1}{h}\int_{\mathbb{R}^3}\abs{\nabla \Tilde{\psi}_h}^2 \ dx\\
        &\leq C.
    \end{align*}

    From \eqref{4.41b} it follows that we can find a constant $C$ such that
    \begin{equation*}
        C\norm{\nabla_h \hat{\psi}_h}_{L^2(\mathbb{R}^3;\mathbb{R}^3)} \leq \norm{M_h^T\nabla_h \hat{\psi}_h}_{L^2(\mathbb{R}^3;\mathbb{R}^3)},
    \end{equation*}
    
    which implies
    \begin{equation*}
        \norm{\nabla_h\hat{\psi}_h}_{L^2(\mathbb{R}^3;\mathbb{R}^3)} \leq C.
    \end{equation*}

    From this, analogously to the proof of \cite[Proposition 4.5]{BrescianiIncompressible}, it follows that
    \begin{equation}
        \hat{\psi}_h \to 0 \ \ \text{in} \ V^{1,2}(\mathbb{R}^3).
        \label{6.31}
    \end{equation}
    It also follows, that there exists $l \in L^2(\mathbb{R}^3)$, such that up to subsequences
    \begin{equation}
        \frac{1}{h}\partial_3\hat{\psi}_h \rightharpoonup l \ \ \text{in} \ L^2(\mathbb{R}^3).
        \label{6.32}
    \end{equation}

    By \eqref{6.27} and then \eqref{5.8}, \eqref{6.31} and \eqref{6.32} together with \eqref{4.42} and \eqref{4.43}, we conclude
    \begin{equation*}
        \begin{aligned}
            \lim_{h\to 0^+}E_h^{mag}(\Tilde{y}_h,\Tilde{m}_h) &= \frac{1}{2}\int_{\mathbb{R}^3}\Tilde{\mu}_h \cdot M_h^T\nabla_h \hat{\psi}_h \ \kappa_h \ dx\\
            &= \int_{\mathbb{R}^3}\chi_\Omega (\lambda)^3 l \ dx.
        \end{aligned}
    \end{equation*}

    \textbf{Step 3} (Identification of the limit):

    To conclude the proof, it is left to show that $l = \chi_\Omega(\lambda)^3$.

    Recall that by Lemma \ref{L 6.1},
    \begin{equation*}
        \Tilde{\psi}_m \in \text{argmin}\big\{ \int_{\mathbb{R}^3}\abs{\nabla \varphi - \chi_{\Omega^{\Tilde{y}_h}}\Tilde{m}_h}^2 \ d\xi : \varphi \in V^{1,2}(\mathbb{R}^3) \big\}.
    \end{equation*}

    Using the Change-of-Variable formula, we obtain
    \begin{equation}
        \Tilde{\psi}_m \in \text{argmin}\big\{ \int_{\mathbb{R}^3}\abs{M_h^T\nabla_h \varphi - \Tilde{\mu}_h}^2\kappa_h \ dx : \varphi \in V^{1,2}(\mathbb{R}^3) \big\}.
        \label{6.34}
    \end{equation}

    The weak formulation of the Euler-Lagrange equations corresponding to \eqref{6.34} is
    \begin{equation}
        \forall \ \varphi \in V^{1,2}(\mathbb{R}^3), \ \int_{\mathbb{R}^3}(M_h^T\nabla_h \hat{\psi}_h - \Tilde{\mu}_h)\cdot M_h\nabla_h \varphi\kappa_h \ dx.
        \label{6.35}
    \end{equation}

    Multiplying \eqref{6.35} by h and letting $h \to 0^+$, again by \eqref{5.8}, \eqref{6.31} and \eqref{6.32} there holds
    \begin{equation}
        \int_{\mathbb{R}^3}(l - \chi_\Omega(\lambda)^3 \partial_3 \varphi \ dx = 0.
        \label{6.36}
    \end{equation}

    Therefore $l - \chi_\Omega(\lambda)^3$ does not depend on $x_3$ as $\varphi$ in \eqref{6.36} is arbitrary. Since $l \in L^2(\mathbb{R}^3)$ and $\chi_\Omega (\lambda)^3 \in L^2(\mathbb{R}^3)$, it follows that $l = \chi_\Omega \lambda$.
    
\end{proof}

 \section{Recovery sequence and Gamma-convergence}
 \label{s:8}

 In this section we prove Theorem \ref{thm2} and Corollary \ref{col1}. For both the existence of a recovery sequence and the proof of the $\Gamma$-convergence result, we adapt \cite[Section 5]{BrescianiIncompressible} to the setting of shallow shells. 

 \subsection*{Recovery sequence for a smooth sequence}

 We follow the strategy of first constructing a smooth recovery sequence and then arguing by density to obtain the desired recovery sequence.
 \begin{theorem}
     (Smooth recovery sequence)
     \label{Thm 7.1}\\
     Let $u \in C^\infty(\overline{\omega};\mathbb{R}^2)$, $v \in C^\infty(\overline{\omega}) $ and $\lambda \in C^\infty(\overline{\omega};\mathbb{S}^2) $. Let $a, b \in C^\infty(\overline{\omega};\mathbb{R}^3) $ satisfy

    \begin{equation*}
        \begin{aligned}
            &\text{tr}\bigg(\left(\begin{array}{c|c}
    	   \text{sym}(\nabla 'u + \nabla ' v \otimes \nabla ' \theta) & 0'\\ 
    	   \hline 
    	   (0')^T & 0
        \end{array}\right)
        + a \otimes e_3 + e_3 \otimes a \bigg)\\
        = \ &\text{div}'(u) + \nabla'v \cdot \nabla'\theta + 2 (a)^3 = 0 \ \ \text{in} \ \omega,
        \end{aligned}
    \end{equation*}

    \begin{equation*}
        \begin{aligned}
            &\text{tr}\bigg(-\left(\begin{array}{c|c}
    	   (\nabla')^2v & 0'\\ 
    	   \hline 
    	   (0')^T & 0
        \end{array}\right) 
        + b \otimes e_3 + e_3 \otimes b \bigg)\\
        = \ &-\Delta'v + 2 (b)^3 = 0 \ \ \text{in} \ \omega.
        \end{aligned}
    \end{equation*}

    Then, there exists a sequence of admissible states $((y_h,m_h))_h \subseteq \mathcal{Q}$ such that, as $h \to 0^+$:
    \begin{align*}
        &y_h \to \Theta_0 \circ z_0 \ \ \text{in} \ W^{1,p}(\Omega;\mathbb{R}^3),\\
        &u_h := U_h^{y_h} \to u \ \ \text{in} \ W^{1,2}(\omega;\mathbb{R}^2),\\
        &v_h := V_h^{y_h} \to v \ \ \text{in} \ W^{1,2}(\omega),
        \label{7.5}\\
        &m_h \circ y_h \to \lambda \ \ \text{in} \ L^{r}(\mathbb{R}^3;\mathbb{R}^3) \ \text{for every} \ 1 \leq r < \infty,\\
        &\mu_h := (\chi_{\Omega^{y_h}}m)\circ \Theta_h \circ z_h \to \chi_\Omega \lambda \ \ \text{in} \ L^{r}(\Omega;\mathbb{R}^3) \ \text{for every} \ 1 \leq r < \infty.
    \end{align*}
    Moreover, there holds:
    \begin{equation*}
        \begin{aligned}
            &\lim_{h \to 0^+}E_h^{el}(y_h,m_h)\\
            &= \frac{1}{2}\int_\omega Q_2^{inc}\bigg(\left(\begin{array}{c|c}
    	\text{sym}(\nabla ' u + \nabla'v \otimes \nabla'\theta) & 0'\\ 
    	\hline 
    	(0')^T & 0
    \end{array}\right) 
    + a \otimes e_3 + e_3 \otimes a, \lambda \bigg) \ dx'\\
    &+ \frac{1}{24}\int_\omega Q_2^{inc}\bigg(-\left(\begin{array}{c|c}
    	(\nabla')^2v & 0'\\ 
    	\hline 
    	(0')^T & 0
    \end{array}\right) 
    + b \otimes e_3 + e_3 \otimes b, \lambda \bigg) \ dx';
        \end{aligned}
    \end{equation*}
    \begin{align*}
        &\lim_{h \to 0^+}E_h^{exc}(y_h,m_h) = E^{exc}(\lambda);\\
        &\lim_{h \to 0^+}E_h^{mag}(y_h,m_h) = E^{mag}(\lambda);
    \end{align*}
 \end{theorem}

 \begin{proof}
     We adapt the proof of \cite[Proposition 5.1]{BrescianiIncompressible} for incompressible magnetoelastic shallow shells.

     \textbf{Step 1} (Gradient of the recovery sequence):
     
     We make the following ansatz for the recovery sequence:
     \begin{equation*}
        \begin{aligned}
            \overline{y}_h := \ &\Theta_h\circ z_h + h^{\frac{\beta}{2}}\begin{pmatrix}u\\0\end{pmatrix} + h^{\frac{\beta}{2}-1}\begin{pmatrix}0'\\v\end{pmatrix} - h^{\frac{\beta}{2}}x_3\begin{pmatrix}\nabla'v\\0\end{pmatrix}\\
            &+ 2h^{\frac{\beta}{2}+1}x_3\Tilde{a} + 2h^{\frac{\beta}{2}+1}x_3^2b.
        \end{aligned}
     \end{equation*}

Using the fact that $\nabla \overline{y}_h = \nabla_h \overline{y}_h \nabla z_h$, we obtain
\begin{equation*}
    \nabla_h \overline{y}_h M_h = \nabla \overline{y}_h (\nabla z_h)^{-1} M_h.
\end{equation*}
Further, there holds
\begin{align*}
    (\nabla z_h)^{-1} M_h &= (\nabla z_h)^{-1}((\nabla \Theta_h)\circ z_h)^{-1} = ((\nabla \Theta_h) \circ z_h \nabla z_h)^{-1} = (\nabla(\Theta_h \circ z_h))^{-1}.
\end{align*}

Together, this yields
\begin{equation}
    \nabla_h \overline{y}_h M_h = \nabla \overline{y}_h (\nabla(\Theta_h \circ z_h))^{-1}.
    \label{7.12}
\end{equation}

Next, we calculate $\nabla_h \overline{y}_h M_h$ up to order $h^{\frac{\beta}{2}}$. We use \eqref{7.12} for the first term to obtain
\begin{equation*}
    \nabla_h(\Theta_h \circ z_h) M_h = \nabla(\Theta_h \circ z_h) (\nabla(\Theta_h \circ z_h))^{-1} = Id.
\end{equation*}

Recall that by \eqref{4.8}, there holds
\begin{equation*}
    M_h = Id + h A \pi + h^2 G_h \circ z_h,
\end{equation*}

with $A(x') = \begin{pmatrix}
        0 & 0 & \partial_1\theta(x')\\
        0 & 0 & \partial_2\theta(x')\\
        -\partial_1\theta(x') & -\partial_2\theta(x') & 0
    \end{pmatrix}$.

Using this we obtain
\begin{equation}
    \begin{aligned}
        \nabla_h \overline{y}_h M_h = \ & Id + h^{\frac{\beta}{2}}\left(\begin{array}{c|c}
    	\nabla ' u + \nabla'v \otimes \nabla'\theta & 0'\\ 
    	\hline 
    	(0')^T & 0
    \end{array}\right)\\
    + &h^{\frac{\beta}{2}-1}\left(\begin{array}{c|c}
    	0'' & -\nabla'v\\ 
    	\hline 
    	\nabla'v^T  & 0
    \end{array}\right) - h^{\frac{\beta}{2}}x_3 \left(\begin{array}{c|c}
    	(\nabla')^2 v & 0'\\ 
    	\hline 
    	(0')^T & 0
    \end{array}\right)\\
    + &2h^{\frac{\beta}{2}}a \otimes e_3 + 2 h^{\frac{\beta}{2}}x_3 b \otimes e_3 + \mathcal{O}(h^{\frac{\beta}{2}+1}),
    \end{aligned}
    \label{7.14}
\end{equation}
\textcolor{Green}{where $a := \Tilde{a} + ((0')^T, \nabla' v \cdot \nabla' \theta)^T$.}\\

\textbf{Step 2} (Incompressibility constraint):

By the incompressibility constraint \eqref{3.4}, $F_h := \nabla_h y_h M_h$ needs to satisfy $\det(F_h) = 1$. To achieve this, we proceed as follows:

Call $\overline{F}_h := \nabla_h \overline{y}_h M_h$ and write $\overline{F}_h = Id + G_h$.

With the fact that
\begin{equation*}
    (Id + G_h)^T(Id + G_h) = Id + 2 \text{sym}G_h + G_h^TG_h,
\end{equation*}

together with $G_h^TG_h = \mathcal{O}(h^{\frac{\beta}{2}+1})$ , as $\frac{\beta}{2} > 3$, we conclude with \eqref{7.14} by a calculation that
\textcolor{Green}{
\begin{equation}
    \overline{F}_h^T\overline{F}_h = Id + 2 h^{\frac{\beta}{2}}(A + x_3 B) + O(h^{\frac{\beta}{2} +1}),
    \label{7.15}
\end{equation}
}
with 
\begin{align*}
    A &= \left(\begin{array}{c|c}
    	\text{sym}(\nabla ' u + \nabla'v \otimes \nabla'\theta) & 0'\\ 
    	\hline 
    	(0')^T & 0
    \end{array}\right) + a \otimes e_3 + e_3 \otimes a,\\
    B &= -\left(\begin{array}{c|c}
    	(\nabla')^2v & 0'\\ 
    	\hline 
    	(0')^T & 0
    \end{array}\right) + b \otimes e_3 + e_3 \otimes b.
\end{align*}

Analogous to \cite[Proposition 5.1]{BrescianiIncompressible}, we use the fact that for every $M \in \mathbb{R}^{3\times 3} $ there holds
\begin{equation*}
    \det (Id + M) = 1 + \text{tr}(M) +  \text{tr}(\text{cof}(M)) + \det (M),
\end{equation*}

to conclude by \eqref{7.15} \textcolor{Green}{together with the assumption $\text{tr}(A) = \text{tr}(B) = 0$} that
\begin{equation*}
    \det(\overline{F}_h^T\overline{F}_h) = 1 + 2h^{\beta}(P + x_3Q + x_3^2R) + \mathcal{O}(h^{\frac{\beta}{2}+1}),
\end{equation*}
where $P, Q$ and $R$ are polynomials in $x'$ which depend on $u, v, \theta, a$ and $b$.

Using $\det(\overline{F}_h^T\overline{F}_h) = \det(\overline{F}_h)^2$ together with the Taylor expansion\\
$(1 + 2x)^{\frac{1}{2}} = 1 + x + \mathcal{O}(x^2)$, we infer
\begin{equation}
    \det(\overline{F}_h) = 1 + h^\beta(P + x_3Q + x_3^2R) + \mathcal{O}(h^{\frac{\beta}{2}+1}).
    \label{7.17}
\end{equation}

We define $y_h(x',x_3) := \overline{y}_h(x',\eta_h(x', x_3))$ for $x \in \Omega$, for a function $\eta_h$ which we will specify later on and set $F_h := \nabla_h y_h M_h$. By writing this in terms of the map $\phi: (x',x_3) \mapsto (x', \eta_h(x',x_3))$ there holds
\begin{align*}
    F_h &= \nabla_h y_h M_h\\
    &= ((\nabla_h \overline{y}_h) \circ \phi)M_h M_h^{-1} \nabla \phi M_h.
\end{align*}

With $\nabla \phi = \begin{pmatrix}
        1 & 0 & 0\\
        0 & 1 & 0\\
        \partial_1\eta_h & \partial_2\eta_h & \partial_3\eta_h
    \end{pmatrix}$,
we conclude that
\begin{equation*}
    \det(M_h^{-1} \nabla \phi M_h) = \det(\nabla \phi) = \partial_3 n_h.
\end{equation*}

To summarize, we have shown that for all $x \in \Omega$ there holds
\begin{equation}
    \begin{aligned}
        &F_h(x',x_3) = \overline{F}_h(x',\eta_h(x',x_3)) M_h^{-1}(x',x_3) \nabla \phi(x',x_3) M_h (x',x_3)\\
    &\det(F_h(x',x_3)) = \det(\overline{F}_h(x', \eta_h(x',x_3)))\partial_3 \eta_h(x',x_3).
    \end{aligned}
     \label{7.20}
\end{equation}

\textcolor{Green}{Following \cite[(5.17)]{BrescianiIncompressible}, by} the fact that $\beta > 2$ and \eqref{7.17}, there holds $\det(\overline{F}_h) = 1 + \mathcal{O}(h^{\frac{\beta}{2} + 1}) $ and therefore \textcolor{Green}{$\frac{1}{2} \leq \det(\overline{F}_h) \leq 2$} for $h$ small enough. With this, we can define
\begin{equation*}
    \Phi_h(x',x_3) := (\det(\overline{F}_h(x',x_3)))^{-1} = (1 + h^\beta(P + x_3Q + x_3^2R) + \mathcal{O}(h^{\frac{\beta}{2}+1}))^{-1},
\end{equation*}
for every $(x',x_3) \in \omega \times (-1,1)$. By \eqref{7.20}, if $\eta_h$ is a solution of the Cauchy problem
\begin{equation}
    \left\{
    \begin{aligned}
        \partial_3 \eta_h(x',x_3) &= \Phi_h(x',\eta_h(x',x_3)), \ \text{in} \ \big(-\frac{1}{2}, \frac{1}{2}\big)\\
        \eta_h(x',0) &= 0,
    \end{aligned}
    \right.
    \label{7.22}
    \end{equation}

then $\det (\nabla_h y_h M_h) = \det(F_h) = 1$ in $\Omega$. By regarding $x' \in \omega$ as a parameter and $x_3 \in (-\frac{1}{2}, \frac{1}{2})$ as a variable, \eqref{7.22} is an ordinary differential equation. 

From \cite[Proposition 5.1, Steps 2 and 3]{BrescianiIncompressible}, we obtain the following two results:

There exists a unique solution $\eta_h$ of \eqref{7.22} satisfying the estimates
\begin{equation}
    \norm{\partial_3 \eta_h - 1}_{C^0(\overline{\Omega})} \leq Ch^{\frac{\beta}{2}+1}, \ \norm{\eta_h - x_3}_{C^0(\overline{\Omega})} \leq Ch^{\frac{\beta}{2}+1}, \ \norm{\nabla'\eta_h}_{C^0(\overline{\Omega})} \leq Ch^{\frac{\beta}{2}+1},
    \label{7.23}
\end{equation}

and there holds
\begin{equation*}
    F_h^T(x)F_h(x) = \overline{F}_h^T(x)\overline{F}_h(x) + \mathcal{O}(h^{\frac{\beta}{2} + 1}).
\end{equation*}

\textbf{Step 3} (Injectivity of the recovery sequence):

By construction, $y_h \in C^1(\overline{\Omega};\mathbb{R}^3)$ and satisfies $\det(\nabla y_h) = \det\nabla(\Theta_h \circ z_h) = h\kappa_h$ by the incompressibilty constraint. Recall that by the ansatz for $\overline{y}_h$ we have
\begin{equation*}
        \begin{aligned}
            y_h = \ &\Theta_h \circ z_h \circ \phi + h^{\frac{\beta}{2}}\begin{pmatrix}u\\0\end{pmatrix} + h^{\frac{\beta}{2}-1}\begin{pmatrix}0'\\v\end{pmatrix} - h^{\frac{\beta}{2}}\eta_h\begin{pmatrix}\nabla'v\\0\end{pmatrix}\\
            &+ 2h^{\frac{\beta}{2}+1}\eta_ha + 2h^{\frac{\beta}{2}+1}\eta_h^2b.
        \end{aligned}
     \end{equation*}

Setting $\varphi_h := y_h - \Theta_h \circ z_h$, we compute 

\textcolor{Green}{\begin{equation*}
    \begin{aligned}
        \nabla_h\varphi_h = \ &S + h^{\frac{\beta}{2}}\left(\begin{array}{c|c}
    	\nabla ' u + \nabla'v \otimes \nabla'\theta & 0'\\ 
    	\hline 
    	(0')^T & 0
    \end{array}\right)\\
    + \ &h^{\frac{\beta}{2}-1}e_3 \otimes \begin{pmatrix}\nabla'v\\0\end{pmatrix} - h^{\frac{\beta}{2}}\begin{pmatrix}\nabla'v\\0\end{pmatrix} \otimes \nabla_h\eta_h - h^{\frac{\beta}{2}}\eta_h \left(\begin{array}{c|c}
    	(\nabla')^2 v & 0'\\ 
    	\hline 
    	(0')^T & 0
    \end{array}\right)\\
    + \ &2h^{\frac{\beta}{2}}a \otimes \nabla_h\eta_h + 2h^{\frac{\beta}{2}+1}\eta_h\left(\begin{array}{c|c}
    	\nabla'a & 0\end{array}\right)\\
    + \ &2 h^{\frac{\beta}{2}}\eta_h b \otimes \nabla_h\eta_h + 2h^{\frac{\beta}{2}+1}\eta_h^2\left(\begin{array}{c|c}
    	\nabla'b & 0\end{array}\right),
    \end{aligned}
\end{equation*}
where, recalling that for $x \in \Omega$, $\Theta_h(z_h(x)) = (x',h \theta(x')) + hx_3n_h(x')$ we wrote
\begin{align*}
    S :=& \ \nabla_h (\Theta_h \circ z_h \circ \phi - \Theta_h \circ z_h)\\
    =& \ h\left(\begin{array}{c|c|c}
    	\partial_1 n_h (\eta_h - x_3) + n_h\partial_1 \eta_h & \partial_2 n_h (\eta_h - x_3) + n_h\partial_2 \eta_h & \frac{1}{h} n_h(\partial_3 \eta - 1)\end{array}\right).
\end{align*}
Using that by Proposition \ref{Prop1}, $\partial_in_h$ is bounded by some constant and the estimates \eqref{7.23}, we obtain
\begin{equation}
    \norm{\nabla_h \varphi_h}_{C^0(\overline{\Omega};\mathbb{R}^{3 \times 3})} \leq Ch^{\frac{\beta}{2} - 1}.
    \label{7.27}
\end{equation}
Setting $\psi_h := \varphi_h \circ (\Theta_h \circ z_h)^{-1}_{|\hat{\Omega}_h}$ together with the fact that $\nabla \Theta^{-1}$ is bounded (see \eqref{4.8}), we find
\begin{equation*}
    \begin{aligned}
        \norm{\nabla \psi_h}_{C^0(\overline{\hat{\Omega}}_h;\mathbb{R}^{3 \times 3})} &= \norm{\nabla \varphi_h \circ (\Theta_h \circ z_h)^{-1} \nabla (\Theta_h \circ z_h)^{-1}}_{C^0(\overline{\hat{\Omega}}_h;\mathbb{R}^{3 \times 3})}\\
        &= \norm{\nabla \varphi_h \circ (\Theta_h \circ z_h)^{-1} (\nabla z_h^{-1}\nabla z_h) \nabla (\Theta_h \circ z_h)^{-1}}_{C^0(\overline{\hat{\Omega}}_h;\mathbb{R}^{3 \times 3})}\\
        &= \norm{\nabla_h \varphi_h \circ (\Theta_h \circ z_h)^{-1} \nabla z_h \nabla (\Theta_h \circ z_h)^{-1}}_{C^0(\overline{\hat{\Omega}}_h;\mathbb{R}^{3 \times 3})}\\
        &\leq Ch^{\frac{\beta}{2}-1}.
    \end{aligned}
\end{equation*}}

With this, analogously to step 4 in the proof of \cite[Proposition 5.1]{BrescianiIncompressible} using the same arguments from \cite[Theorem 5.5-1]{CiraletElasticity}, injectivity of $y_h$ follows.

Steps 5 and 6 in the proof of \cite[Proposition 5.1]{BrescianiIncompressible} finish this proof by providing the necessary convergences and require no adaptation.
     
 \end{proof}
With Theorem \ref{Thm 7.1}, the proof of Theorem \ref{thm2} follows by the same density argument as conducted in \cite[Theorem 3.2]{BrescianiIncompressible} and will therefore be omitted.

\subsection*{Gamma-convergence}

We conclude this section by showing the $\Gamma$-convergence result stated in Corollary \ref{col1}. The proof relies on our two main results, namely, compactness and lower bound, as well as the construction of recovery sequences. We present an adaptation of the proof of \cite[Corollary 3.1]{BrescianiIncompressible} for shallow shells.

\begin{proof}[Proof of Corollary \ref{col1}]
    For most parts we refer to the proof of \cite[Corollary 3.1]{BrescianiIncompressible}. We will only conduct the argument which requires an adaptation for the setting of shallow shells.

    Consider the translation motion $\check{T}_h$ defined by 
    \begin{equation*}
        \check{T}_h(\xi) := \xi - \check{c}_h,
    \end{equation*}
    for every $\xi \in \mathbb{R}^3$, where
    \begin{equation*}
        \check{c}_h := \fint_\Omega (y_h - \Theta_h\circ z_h) \ dx.
    \end{equation*}
    We set $\check{y}_h := \check{T}_h \circ y_h = y_h - \check{c}_h$ and $\check{m}_h := m_h \circ \check{T}_h^{-1}$. By arguments from the proof of Theorem \ref{Thm 5.1}, it follows that
    \begin{equation*}
    \begin{aligned}
        &\check{u}_h := U_h^{\check{y}_h} \rightharpoonup \check{u} \ \text{in} \ W^{1,2}(\omega;\mathbb{R}^2),\\
        &\check{v}_h := V_h^{\check{y}_h} \rightharpoonup \check{v} \ \text{in} \ W^{1,2}(\omega),
    \end{aligned}
    \end{equation*}
    for some $\check{u} \in W^{1,2}(\omega;\mathbb{R}^2)$ and $\check{v} \in W^{1,2}(\omega)$. Then, by an argument from the proof of Theorem \ref{Thm 5.2} we find $\check{\lambda} \in W^{1,2}(\omega;\mathbb{S}^2)$, such that
    \begin{equation*}
        \begin{aligned}
            &\check{m}_h \circ \check{y}_h \to \lambda \ \text{in} \ L^r(\Omega;\mathbb{R}^3) \ \text{for every} \ 1 \leq r < \infty,\\
            &\check{\mu}_h := \mathcal{M}_h(\check{y}_h, \check{m}_h) \to \chi_\Omega\check{\lambda} \ \text{in} \ L^r(\Omega;\mathbb{R}^3) \ \text{for every} \ 1 \leq r < \infty.
        \end{aligned}
    \end{equation*}
    It remains to show that $\check{\lambda} = \lambda$ holds. For this, we first deduce analogously to \cite{BrescianiIncompressible},
    \begin{equation}
        \abs{\check{c}_h} \leq Ch^{\frac{\beta}{2}-1}.
        \label{7.31}
    \end{equation}
    Let $\varphi \in C_c^\infty(\omega;\mathbb{R}^3)$ and call $\overline{\varphi}$ its extension to the whole space by zero. By the Change-of-Variable formula and the identity $\chi_{\Omega^{\check{y}_h}}\check{m}_h = (\chi_{\Omega^{y_h}}m_h) \circ \check{T}_h^{-1}$, we obtain
    \begin{equation}
        \begin{aligned}
            \int_{\mathbb{R}^3}\check{\mu}_h\cdot\overline{\varphi} \ \kappa_h \ dx &= \int_{\mathbb{R}^3}(\chi_{\Omega^{y_h}}m_h) \circ \check{T}_h^{-1}\circ \Theta \circ z_h \cdot \overline{\varphi} \ \kappa_h \ dx\\
            &= \frac{1}{h}\int_{\mathbb{R}^3}\chi_{\Omega^{y_h}}m_h\cdot \overline{\varphi} \circ (\Theta_h \circ z_h)^{-1} \circ \check{T}_h \ d\xi\\
            &= \int_{\mathbb{R}^3} \mu_h \cdot \overline{\varphi} \circ (\Theta_h \circ z_h)^{-1} \circ \check{T}_h \circ \Theta_h \circ z_h \ dx.
        \end{aligned}
        \label{7.32}
    \end{equation}
    Notice that
    \begin{equation}
        \begin{aligned}
            (\Theta_h \circ z_h)^{-1} \circ \check{T}_h \circ \Theta_h \circ z_h (x) &= (\Theta_h \circ z_h)^{-1} \circ \check{T}_h(\xi)\\
            &= (\Theta_h \circ z_h)^{-1} (\xi - \check{c}_h)\\
            &= x - (\Theta_h \circ z_h)^{-1}(\check{c}_h).
        \end{aligned}
        \label{7.33}
    \end{equation}
    By property \eqref{7.31}, the fact that $(\Theta_h \circ z_h)^{-1} \to Id$ uniformly and \eqref{7.33}, we deduce 
    \begin{equation*}
        (\Theta_h \circ z_h)^{-1} \circ \check{T}_h \circ \Theta_h \circ z_h (x) \to x,
    \end{equation*}
    for every $x \in \Omega$. Therefore, by the Dominated Convergence Theorem, we obtain
    \begin{equation*}
        \overline{\varphi} \circ (\Theta_h \circ z_h)^{-1} \circ \check{T}_h \circ \Theta_h \circ z_h \to \overline{\varphi} \ \text{in} \ L^2(\mathbb{R}^3;\mathbb{R}^3).
    \end{equation*}
By Theorem \ref{Thm 5.2}, we obtain 
\begin{equation*}
    \mu_h := \mathcal{M}_h(y_h,m_h) \to \chi_\Omega \lambda \ \text{in} \ L^r(\mathbb{R}^3;\mathbb{R}^3) \ \text{for every} \ 1 \leq r < \infty,
\end{equation*}
for some $\lambda \in W^{1,2}(\omega;\mathbb{S}^2)$. Furthermore, there holds $\mu = \chi_\Omega\lambda$.

All combined, together with \eqref{4.42}, by passing to the limit $h \to 0^+$ in \eqref{7.32}, we obtain
\begin{equation*}
    \int_\omega\check{\lambda}\cdot\varphi \ dx' = \int_\Omega \check{\lambda} \cdot \overline{\varphi} \ dx = \int_\Omega\lambda \cdot \overline{\varphi} \ dx = \int_\omega \lambda \cdot \varphi \ dx'.
\end{equation*}

As $\varphi$ is arbitrary, we conclude that $\check{\lambda} = \lambda$.

The rest of the proof from \cite[Corollary 3.1]{BrescianiIncompressible} requires no adaptation and is therefore omitted.
\end{proof}
\section*{Acknowledgements} 

This research has been supported by the Austrian Science Fund (FWF) through grants \href{https://doi.org/10.55776/Y1292}{10.55776/Y1292} and \href{https://doi.org/10.55776/F100800}{10.55776/F100800}. For open access purposes, the authors have applied a CCBY public copyright license to any accepted manuscript version arising from this submission.

\bibliographystyle{siam}
\bibliography{main}

\end{document}